\documentclass{gtart}

\usepackage{amsmath, amssymb, latexsym, amscd, graphicx, epstopdf}
\usepackage{euscript}
\usepackage{subfigure}
\usepackage{epsfig}
\usepackage{psfrag}

\usepackage[small,bf]{caption}
\usepackage[inner=25mm, outer=25mm, textheight=225mm]{geometry}

\def\bkC{{\rm \kern.24em \vrule width.05em height1.4ex depth-.05ex 
\kern-.26em C}}

\def\bksC{{\rm \kern.24em \vrule width.05em height1ex depth-.05ex 
\kern-.26em C}}

\def\bkE{{\rm I\kern-.22em E}}
 
\def\bkH{{\rm I\kern-.22em H}}
 
\def\bkN{{\rm I\kern-.17em N}}

\def\bkQ{{\rm \kern.24em \vrule width.05em height1.4ex depth-.05ex 
\kern-.26em Q}}
\def\Q{\bkQ}
\def\bkR{{\rm I\kern-.17em R}}
\def\RR{\bkR}

\def\bkZ{{\rm Z\kern-.32em Z}}
\def\Z{\bkZ}
\def\bksZ{{\rm Z\kern-.22em Z}}

\def\tri{\mathcal{T}}

\DeclareMathOperator{\LST}{LST}

\theoremstyle{plain}

\newtheorem{thm}{Theorem}
\newtheorem*{thm*}{Theorem}
\newtheorem{lem}[thm]{Lemma}
\newtheorem*{lem*}{Lemma}
\newtheorem{cor}[thm]{Corollary}
\newtheorem*{cor*}{Corollary}
\newtheorem{pro}[thm]{Proposition}
\newtheorem*{pro*}{Proposition}
\newtheorem{rem}[thm]{Remark}
\newtheorem*{rem*}{Remark}
\newtheorem{defn}[thm]{Definition}
\newtheorem*{defn*}{Definition}
\newtheorem{rmk}[thm]{Remark}
\newtheorem*{rmk*}{Remark}

\newtheorem*{cla*}{Claim}

\addtocounter{MaxMatrixCols}{4}
\begin{document}
\title{A new combinatorial class of 3--manifold triangulations} 
\author{Feng Luo and Stephan Tillmann}
\begin{abstract}
We define a new combinatorial class of triangulations of closed 3--manifolds, satisfying a weak version of 0--efficiency combined with a weak version of minimality, and study them using twisted squares. As an application, we obtain strong restrictions on the topology of a 3--manifold from the existence of non-smooth maxima of the volume function on the space of circle-valued angle structures.
\end{abstract}
\primaryclass{57M25, 57N10}
\keywords{3--manifold, triangulation, circle-valued angle structure}


\maketitle



\section{Introduction}

In computational topology, it is a difficult problem to certify that a given triangulation of a closed, irreducible, orientable 3--manifold is minimal. Here, the term \emph{triangulation} includes semi-simplicial or singular triangulations. The difficulty mainly stems from the fact that on the one hand, the current computer generated censuses are limited in size (see \cite{Mat2007, Regina}), and on the other hand, it is difficult to find good lower bounds for the minimal number of tetrahedra (see \cite{Mat2003, JRT1}). For many algorithms using 3--manifold triangulations, the 0--efficient triangulations due to Jaco and Rubinstein~\cite{JR} have been established as a suitable platform. However, certifying that a given triangulation is 0--efficient involves techniques from linear programming. 

This paper introduces a new class of triangulations satisfying a weak version of 0--efficiency combined with a weak version of minimality. The weak version of minimality (\emph{face-pair-reduced}) is that certain simplification moves, which can be determined from the 2--skeleton, are not possible. The weak version of 0--efficiency (\emph{face-generic})  puts mild restrictions on the combinatorics of the 2--skeleton.
We show that a 0--efficient or minimal triangulation of a closed, irreducible, orientable 3--manifold is face-generic unless the underlying manifold is one of $S^3,$ $\RR P^3,$ $L(3, 1),$ $L(4, 1),$ $L(5, 1)$ or $L(5,2).$ Moreover, if a triangulation is not face-generic and face-pair-reduced, then it can either be simplified to a face-generic, face-pair-reduced triangulation or it is easy to recognise which of the aforementioned six exceptions is the underlying manifold of the triangulation (Proposition~\ref{pro:algo to simplify triangulation}).

The new notions, face-pair-reduced and face-generic, are motivated by key combinatorial properties of the 2--skeleton of minimal and 0--efficient triangulations respectively. They are the subject of Section~\ref{sec:faces}. A new tool in the study of triangulations is introduced in Section~\ref{sec:squares}: the study of (possibly pinched or immersed) surfaces that have a cell structure consisting of up to two squares, each of which is mapped to a \emph{twisted square} in a tetrahedron. Twisted squares are also used to give a classification of the possible combinatorial types of the (possibly singular) 3--simplices in a face-pair-reduced, face-generic triangulation. In this paper, as is common usage, we call a singular 3--simplex in $M$ a tetrahedron.

We give a key application to highlight the utility of our techniques. Previous work of the first author~\cite{Luo-2009} defined and studied circle-valued angle structures on orientable, closed 3--manifolds using a volume function. It was shown that a smooth maximum of the volume function corresponds to a solution of the \emph{generalised Thurston equation}, whilst a non-smooth maximum corresponds to a \emph{cluster of three 2--quad type solutions} to Haken's normal surface equations. Such a cluster consists of three algebraic solutions with at most two non-zero quadrilateral coordinates, such that the set of their non-zero coordinates includes the three quadrilaterals supported by some tetrahedron.
The present paper shows that (under sufficiently strong hypotheses on the triangulation), the existence of a non-smooth maximum gives strong information about the topology of the manifold. The following is a consequence of the more technical Theorem~\ref{thm:cluster}, which is stated and proved in Section~\ref{sec:Clusters of three 2--quad type solutions}:

\begin{thm}\label{thm:main}
Suppose $M$ is a closed, orientable 3--manifold with a triangulation which is either minimal or (face-pair-reduced and face-generic). If there is a cluster of three 2--quad type solutions to Haken's normal surface equations, then either
\begin{enumerate}
\item $M$ is reducible, or
\item $M$ is toroidal, or
\item $M$ is Seifert fibred, or
\item $M$ contains a non-orientable surface of genus 3.
\end{enumerate}
\end{thm}

A discussion of this theorem and its consequences can be found in \cite{Luo-2011} (\S1 and \S4.3), where it was first announced for minimal triangulations. In particular, Thurston's orbifold theorem implies that if $M$ does not fall in the first three cases and contains a non-orientable surface of genus 3, then $M$ satisfies Thurston's Geometrisation Conjecture without appeal to Perelman's Ricci flow techniques. The proof of the theorem shows that, for a 3--manifold with face-pair-reduced and face-generic triangulation, it is possible to deduce from the combinatorics of the triangulation which of the (not mutually exclusive) cases occurs (see Lemma~\ref{lem:1-quad type}, Proposition~\ref{pro:2-quad type in 1 tet} and Theorem~\ref{thm:cluster}). As pointed out by the referee, it is worth mentioning that the main application of this paper can be viewed as a step towards realising a variant of Casson's program to prove the Geometrisation Theorem for 3--manifolds. Casson's program originally aimed to reprove Thurston's theorem via ideal triangulations in the case of 3--manifolds with boundary consisting of tori, and this program was later extended to closed 3--manifolds by the first author.

We have developed the new tools and notions for triangulations with a view towards keeping this paper brief, and the details at bay. They can be modified for the study of reducible 3--manifolds. Moreover, there is much scope in developing a full theory of face-pair-reduced and face-generic triangulations of irreducible 3--manifolds as they share many combinatorial properties with 0--efficient and minimal triangulations. This would, in particular, give new information about minimal triangulations of closed hyperbolic 3--manifolds. These remain elusive at the time of writing: they are sparse in the current censuses and no infinite families are known.
  
\textbf{Acknowledgements} 
The first author is partially supported by the USA National Science Foundation (project numbers DMS 1105808 and DMS 1222663). The second author is partially supported under the Australian Research Council's Discovery funding scheme (project number DP110101104), and thanks the Max Planck Institute for Mathematics, where parts of this work have been carried out, for its hospitality. The authors would like to thank the referee for constructive feedback.
  

\newpage

\section{Faces in triangulations}
\label{sec:faces}

The properties of triangulations that we will need are a weak version of 0--efficiency combined with a weak version of minimality. The weak version of minimality (\emph{face-pair-reduced}) is that certain simplification moves, which can be determined from the 2--skeleton, are not possible. The weak version of 0--efficiency (\emph{face-generic}) puts mild restrictions on the combinatorics of the 2--skeleton. 


\subsection{Triangulations and normal surfaces}

We will assume that the reader is familiar with standard definitions, facts and notions about triangulations of 3--manifolds and normal surfaces. 
Our main references are \cite{JR} for triangulations of 3--manifolds and \cite{tillus-normal} for normal surface theory. A triangulation of a compact, orientable 3--manifold $M$ consists of a union of pairwise disjoint 3--simplices, $\widetilde{\Delta},$ a set of face pairings, $\Phi,$ and a natural quotient map $p\co \widetilde{\Delta} \to \widetilde{\Delta} / \Phi = M.$ Since the quotient map is injective on the interior of each 3--simplex, we refer to the image of a 3--simplex in $M$ as a \emph{tetrahedron} and to its faces, edges and vertices with respect to the pre-image. Similarly for images of 2-- and 1--simplices, which will be referred to as \emph{faces} and \emph{edges} in $M.$ 

A triangulation of the closed, orientable, connected 3--manifold $M$ is \emph{minimal} if $M$ has no other triangulation with fewer tetrahedra. A triangulation of $M$ is  \emph{0--efficient} if the only embedded, normal 2--spheres are vertex linking (see \cite{JR} or Section~\ref{sec:Extremal rays with support in one tetrahedron} for a definition of normal surface).


\subsection{Small triangulations}

If the closed, irreducible, orientable 3--manifold $M$ has a (not necessarily minimal) triangulation with just one or two tetrahedra, then $M$ is one of the lens spaces $S^3,$ $\RR P^3,$ $L(3,1),$ $L(4,1),$ $L(5, 1),$ $L(5, 2),$ $L(7, 2),$ $L(8, 3)$ or quaternionic space $S^3/Q_8$ (see, for instance, the census in \cite{Regina}). 
Given such a small triangulation, one can recognise $M$ from the combinatorics of the triangulation or (up to the ambiguity $L(5, 1)$ versus $L(5, 2)$) from $H_1(M).$

A subset of the above list of manifolds of small complexity turns out to play a special role in the analysis of the combinatorics of the 2--skeleton of a triangulation.  For simplicity, we will write $L(p)$ for a lens space with fundamental group $\Z_p$; this mainly saves some space that would be taken up with distinguishing $L(5,1)$ from $L(5,2).$


\subsection{Faces in triangulations}
\label{subsec:Faces in triangulations}

\captionsetup[subfigure]{labelformat=empty}
\begin{figure}[h]
  \begin{center}
    \subfigure[triangle]{\includegraphics[height=2.1cm]{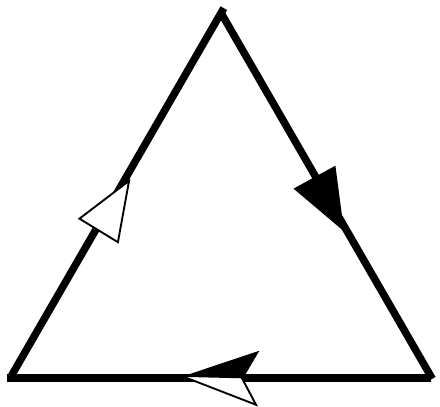}}
    \qquad
     \subfigure[cone]{\includegraphics[height=2.1cm]{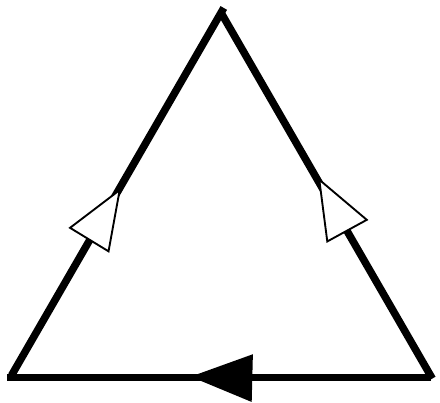}}
  \qquad
     \subfigure[M\"obius]{\includegraphics[height=2.1cm]{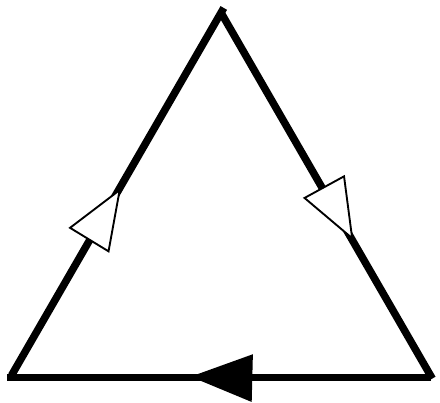}}
  \qquad
     \subfigure[3--fold]{\includegraphics[height=2.1cm]{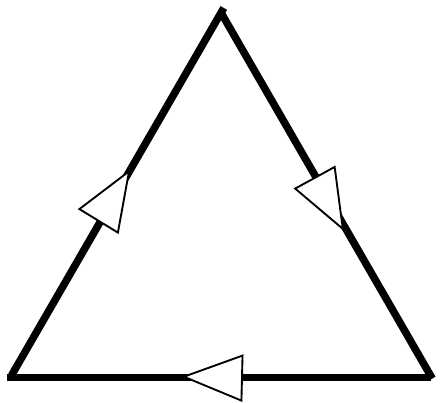}}
     \qquad
     \subfigure[dunce]{\includegraphics[height=2.1cm]{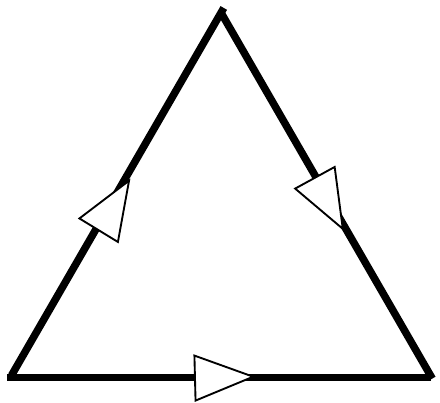}}
          \end{center}
   \caption{Types of faces. The M\"obius, 3--fold and dunce faces have only got one vertex due to the edge identifications. A cone face may have one or two vertices and a triangle face may have up to 3 vertices.}
   \label{fig:faces}
\end{figure}
\captionsetup[subfigure]{labelformat=default}

There are eight types of faces in a triangulation, depending on how vertices or edges are identified (see Figure 1 in \cite{JR}). Ignoring possible identification of vertices, this reduces to the following five types based on the edge identifications only (see Figure~\ref{fig:faces} below). A face with no edge-identifications is termed a \emph{triangle face}; it is a 2--simplex with some or all vertices identified. If a pair of edges is identified, the face is either a cone (possibly with the tip identified with a point on the boundary) and called a \emph{cone face} or it is a M\"obius band and called a \emph{M\"obius face}. In a M\"obius face, we distinguish the \emph{boundary edge} and the \emph{core edge}. If all three edges are identified, the face is either a 3--fold or a dunce hat, and called a \emph{3-fold face} or \emph{dunce face} respectively.

The following results concerning minimal, 0--efficient or arbitrary triangulations are well-known:

\begin{thm}[\cite{JR}, Theorem 6.1]
A minimal triangulation of the closed, irreducible, orientable 3--manifold $M$ is 0--efficient unless $M = \RR P^3$ or $L(3,1).$
\end{thm}

\begin{pro}\label{pro:0-efficient summary}
Suppose $M$ is a closed, orientable 3--manifold with a 0--efficient triangulation.
\begin{enumerate}
\item $M$ irreducible and not $\RR P^3.$ (\cite{JR} Proposition 5.1)
\item The triangulation has one vertex or $M=S^3$ and there are precisely 2 vertices. (\cite{JR} Proposition 5.1)
\item If an edge bounds an embedded disc, then $M=S^3.$ (\cite{JR}, Proposition 5.3)
\item No face is a cone unless $M=S^3.$ (\cite{JR}, Corollary 5.4)
\end{enumerate}
\end{pro}

\begin{pro}
Suppose $M$ is a closed, orientable, irreducible 3--manifold with an arbitrary triangulation. If a face in the triangulation is a 3--fold, then $M=L(3).$
\end{pro}

\begin{proof}
Let $N$ be a regular neighbourhood of the 3--fold $F$ in $M.$ Then $2 = 2\; \chi(F) = 2\; \chi(N) = \chi(\partial N).$ If $\partial N$ is connected, then it is a sphere and the result follows since $M$ is irreducible and the 3--fold is then a spine for $M.$ To see that $\partial N$ must be connected, equip the interior of the triangle that makes up the 3-fold with a transverse orientation. The single edge in the boundary of the 3--fold is a simple closed curve $\gamma$ in $M.$ Let $p\in \gamma$ and $B_p$ be a small regular neighbourhood of $p.$ Then $F \cap B_p$ looks like the shape of a $Y$ times an interval, and there must be at least one component of $B_p \setminus (F \cap B_p)$ with the property that the transverse orientation points both into the component and out of it. But this implies that $F$ cannot separate $N,$ whence $N$ has connected boundary.
\end{proof}

The starting point for our study of triangulations is the following:

\begin{lem}
Suppose $M$ is a closed, orientable, irreducible 3--manifold with an arbitrary triangulation, and $\sigma$ a tetrahedron in $M.$ Suppose three faces of $\sigma$ are M\"obius faces, and the remaining face is either a M\"obius face or a triangle face. Then $M=L(4)$ or $L(5).$
\end{lem}

\begin{proof}
Suppose there are three M\"obius faces and one triangle face in $\sigma.$ First suppose that each edge of the triangle face is the core edge of a M\"obius face. Then there is (up to symmetry) only one possibility, and one sees that the vertex link cannot be a sphere. Hence some edge of the triangle face is the boundary edge of a M\"obius face. If the core edge of this M\"obius face is not identified with one of the other edges of the triangle face, then not both of the remaining faces can be M\"obius faces. Hence it is identified with one of the other edges of the triangle face. Up to symmetry, there is again only one possibility and one finds a spine for $L(5)$ made up of one \emph{twisted square} (see \S\ref{subsec:Squares in triangulations} for a definition) and one face.

Hence suppose there are four M\"obius faces in $\sigma.$ Given two M\"obius faces, they meet along some edge $e$ simply as a result of being faces of $\sigma.$ Suppose $e$ is a boundary edge of both M\"obius faces. Then the condition that the other two faces are also M\"obius forces the core edges to be identified such that a square is a spine for $L(4).$ Enumerating the remaining cases, one finds only one possibility up to symmetry, and this gives a spine for $L(5).$
\end{proof}

The following result follows from \cite{JR} (Proposition 5.3 and Corollary 5.4) and \cite{JRT} (Lemma 7); see also the proof of Proposition~\ref{pro:algo to simplify triangulation} below.

\begin{lem}\label{lem:simplify cone faces}
Suppose $M$ is a closed, orientable, irreducible 3--manifold with an arbitrary triangulation. 
If a face in the triangulation is a cone or a dunce hat, then the triangulation can be simplified to have fewer tetrahedra or $M$ is homeomorphic with one of $S^3,$ $\RR P^3,$ or $L(3).$
\end{lem}

The above observations lead us to the following definition:

\begin{defn}[face-generic]
Suppose $M$ is a closed, orientable 3--manifold. We say that a triangulation of $M$ is \emph{face-generic} if all faces are triangles or M\"obius bands and each tetrahedron has at most two M\"obius faces.
\end{defn}

\begin{rem}
We have the following immediate consequences:
\begin{enumerate}
\item A face-generic triangulation may have more than one vertex. Examples are simplicial triangulations of any manifold, or the natural 2--vertex triangulations of lens spaces.
\item A face-generic triangulation has no edge of degree one, since such an edge can only be obtained from folding two faces of a tetrahedron together and hence creating at least one face that is a cone or dunce hat.
\item A 0--efficient triangulation of $M$ is face-generic unless $M$ is homeomorphic with one of $S^3,$ $L(3),$ $L(4)$ or $L(5).$ 
\item A minimal triangulation of the closed, orientable, irreducible 3--manifold $M$ is face-generic unless $M$ is one of $S^3,$ $\RR P^3,$ $L(3),$ $L(4)$ or $L(5).$ 
\end{enumerate}
\end{rem}

Whilst it is easy to recognise whether or not a triangulation is face-generic, the class of face-generic triangulations is too large for many purposes. We now turn to pairs of faces to further restrict the class of triangulations we wish to consider.


\subsection{Pairs of faces in triangulations}
\label{subsec:Pairs of faces in triangulations}

A combinatorial map will mean a continuous map between pseudo-manifolds $N\to M$ sending cells to cells and which restricted to the interior of each (possibly singular) simplex of each dimension is a homeomorphism onto its image. In particular, $\dim N \le \dim M.$ An example is the map $\widetilde{\Delta} \to M.$

\begin{figure}[h]
  \begin{center}
    \subfigure[$D^2$]{\includegraphics[height=2.8cm]{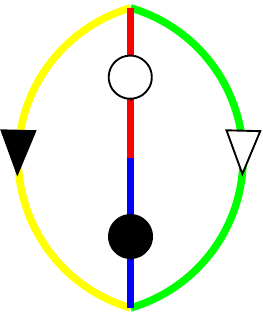}}
    \qquad\qquad
     \subfigure[$S^2$]{\includegraphics[height=2.8cm]{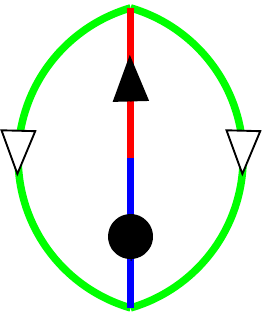}}
  \qquad\qquad
     \subfigure[$P^2$]{\includegraphics[height=2.8cm]{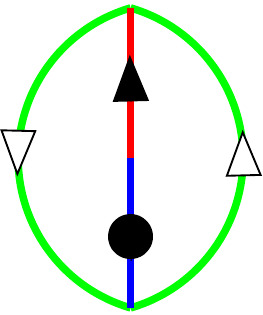}}
  \\
     \subfigure[$D^2$]{\includegraphics[height=2.8cm]{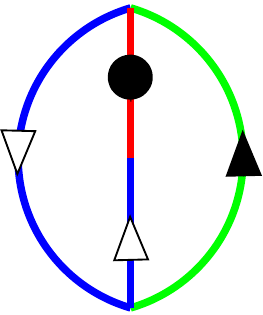}}
     \qquad\qquad
     \subfigure[$D^2$]{\includegraphics[height=2.8cm]{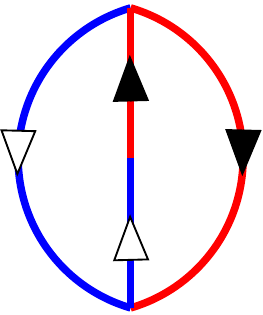}}
      \qquad\qquad
     \subfigure[$S^2/T^2$]{\includegraphics[height=2.8cm]{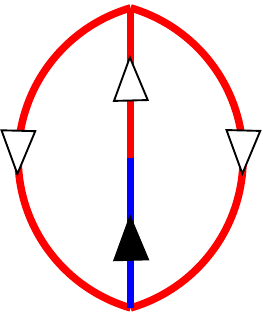}}
          \end{center}
   \caption{Identifications of two distinct faces along at least two edges: Picture (c) does not occur, and all other pictures allow an explicit modification of the triangulation making it smaller. The captions indicate which closed surfaces (if any) the map factors through.}
   \label{fig:Disc subcomplexes}
\end{figure}

The following lemma follows from results by Burton \cite{BAB-2004, BAB-2007}.

\begin{lem}[Disc subcomplexes]\label{lem:disc subcomplex}
Let $M$ be a closed, orientable, irreducible 3--manifold with a face-generic triangulation having at least 3 tetrahedra.
Let $D$ be a disc triangulated with 2 triangles meeting along two edges, and suppose that $f\co D \to M$ is a combinatorial map with the property that the images of the two triangles in $D$ are distinct faces in $M.$
\begin{enumerate}
\item If the two interior edges of $D$ are mapped to distinct edges, then the triangulation can be modified to a triangulation of $M$ having fewer tetrahedra.
\item If the two interior edges of $D$ are mapped to the same edge, then the triangles in $D$ are mapped to M\"obius faces with distinct boundary edges, or the two faces form a spine for $L(4,1)$ and hence $M=L(4,1),$ or the triangulation can be modified to a triangulation of $M$ having fewer tetrahedra.
\end{enumerate}
\end{lem}

\begin{proof}
The first part follows directly from Lemma 2.7 and Corollary 2.10 in \cite{BAB-2004}, and Lemma 3.6 and Corollary 3.8 in \cite{BAB-2007}.

For the second part, observe that if the two interior edges are identified, then, because we cannot get cones, dunce hats or 3-folds, the images are M\"obius faces sharing at least the core edges. Suppose the two boundary edges are identified. There are two cases, depending on the orientation of the edges. One case gives a spine for $L(4,1).$ In the other case, the map has image  equivalent to Figure~\ref{fig:Disc subcomplexes}(f) and hence there is a map $f'\co D\to M$ as in part (1).
\end{proof}

\begin{rmk}
The techniques of the lemma also apply to reducible manifolds. In this case, the modification procedure
either results in a triangulation of $M$ having fewer tetrahedra, or $M = M_1\# M_2,$ and one obtains triangulations of $M_1$ and $M_2,$ each having fewer tetrahedra.
\end{rmk}

The above lemma leads us to the following definition  (see also Figure~\ref{fig:Face-pair-reduced}):

\begin{defn}[face-pair-reduced]\label{defn:face-pair-reduced}
Let $D$ be a disc triangulated with 2 triangles meeting along two edges. We will call a triangulation \emph{face-pair-reduced} if 
every combinatorial map $f\co D \to M,$ which maps the two triangles in $D$ to distinct faces in $M,$ has the property that the two interior edges of $D$ are mapped to the same edge, and the triangles in $D$ are mapped to M\"obius faces with distinct boundary edges.
\end{defn}

\begin{figure}[h]
  \begin{center}
     \includegraphics[height=3.9cm]{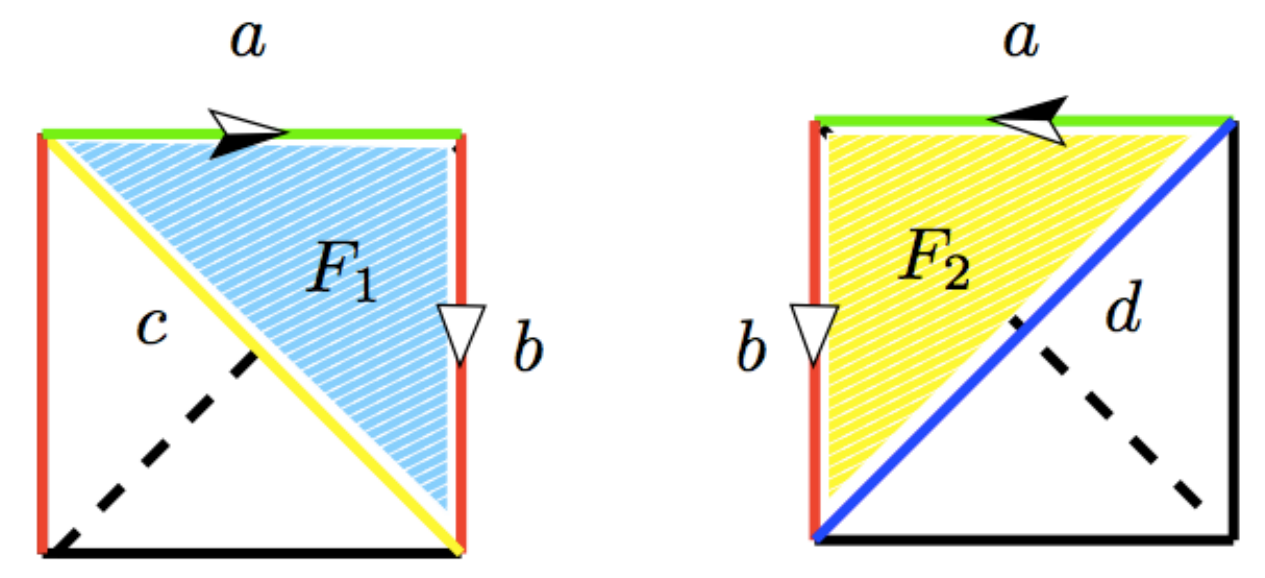}            
  \end{center}
 \caption{Face-pair-reduced. If $F_1\neq F_2,$ then $a=b$ and $c\neq d.$}
 \label{fig:Face-pair-reduced}
\end{figure}

\begin{rem}\label{rem: cone face implies not face-pair-reduced}
We have the following immediate consequences:
\begin{enumerate}
\item A tautological, but most useful, consequence of the definition is that case (1) in Lemma~\ref{lem:disc subcomplex} cannot occur in a face-generic, face-pair-reduced triangulation. Thus, each face is uniquely characterised by two consecutive distinct, oriented edges in its boundary.
\item A face-pair-reduced triangulation may have more than one vertex. Examples are simplicial triangulations of any manifold, and the natural 2--vertex triangulations of lens spaces.
\item A face-pair-reduced triangulation has no cone faces (see Burton~\cite{BAB-2004}, Lemma 2.8).
\item A minimal triangulation of a closed, orientable, irreducible 3--manifold with at least 3 tetrahedra is face-pair-reduced.
\item A 0--efficient and face-generic triangulation may not be face-pair-reduced. An example is given by the 4 tetrahedron triangulation of $S^3$ defined by the following face pairings:

\begin{center}
\begin{tabular}{l | l | l | l | l }
Tetrahedron & Face 012 & Face 013 & Face 023 & Face 123 \\
\hline
    0  &               2 (231)  &  1 (321)  &  0 (312) &   0 (230)\\
    1  &              3 (231) &   3 (023)  &  2 (032)  &  0 (310)\\
    2  &              3 (013)  &  3 (120)  &  1 (032) &   0 (201)\\
    3  &              2 (301)  &  2 (012)  &  1 (013)  &  1 (201)\\
\end{tabular}    
\end{center}    

This triangulation can be simplified to a 1--tetrahedron triangulation of $S^3$ via a sequence of three Pachner moves.
\end{enumerate}
\end{rem}


\subsection{Obtaining small face-pair-reduced, face-generic triangulations}

Every simplicial triangulation is face-pair-reduced and face-generic. In particular, one can convert any triangulation to a face-pair-reduced, face-generic triangulation by performing at most two barycentric subdivisions. However, this increases the number of tetrahedra.

\begin{pro}\label{pro:algo to simplify triangulation}
There is an algorithm which takes as input any triangulation of a closed, orientable, irreducible 3--manifold $M,$ and either outputs a face-pair-reduced, face-generic triangulation of $M$ having at most the same number of tetrahedra and precisely one vertex, or concludes that $M$ is $S^3,$ $\RR P^3,$ $L(3),$ $L(4)$ or $L(5).$
\end{pro}

\begin{proof}
For the purpose of this paper, we are merely interested in the existence of an algorithm, not in producing an efficient algorithm.
\begin{enumerate}
\item[(0)] First run the 3--sphere recognition algorithm to decide whether $M$ is $S^3.$ If yes, output $M$ is $S^3.$ Else go to step (1).
\item[(1)] If the triangulation has only one vertex, go to (2). Otherwise the given triangulation has more than one vertex, and one converts it to a 1--vertex triangulation, thereby reducing the number of tetrahedra (see the proofs of Theorem 5.5 and Proposition 5.1 in \cite{JR}). Then go to (2).
\item[(2)] If there are at most 2 tetrahedra, one recognises the manifold and if one exists, outputs a face-pair-reduced, face-generic, 1--vertex triangulation of $M,$ or outputs the homeomorphism type of $M.$ Otherwise go to step (3).
\item[(3)] If a face is a 3--fold, then output $M$ is $L(3)$; else if some tetrahedron has more than 2 M\"obius faces, then $M=L(4)$ or $L(5),$ and the conclusion can be obtained either from the combinatorics of the 2--skeleton, or by computing $H_1(M)$ from the 2--skeleton; else go to step (4).
\item[(4)] Check if the triangulation is face-pair-reduced (recall that existence of cone or dunce faces implies not face-pair-reduced). If ``Yes," then the triangulation is face-pair-reduced and face-generic, and output the triangulation. If ``No," then apply the explicit operation from the proof of Lemma 2.7 in \cite{BAB-2004}. This either gives a triangulation of $M$ having the same number of tetrahedra but more vertices, or it gives triangulations of $M$ and $S^3,$ each having fewer tetrahedra. In the second case, one needs to run the 3--sphere recognition algorithm to determine which triangulation corresponds to $M.$ In either case, one feeds the new triangulation of $M$ back to (1).
\end{enumerate}
In steps (2) and (3), the number of tetrahedra is not affected, so if the algorithm does not terminate in these steps, one arrives at (4) with the same number of tetrahedra. In step (4), the number of tetrahedra either decreases and one goes back to (1), or the number of tetrahedra remains the same, but subsequently decreases in step (1). In particular, if the algorithm does not terminate in steps (2), (3) or (4), then the loop $(4) \to (1) \to (2) \to (3) \to (4)$ decreases the number of tetrahedra. It follows that the algorithm will terminate.
\end{proof}

\begin{rem}
For a reducible 3--manifold, $M,$ the above algorithm can be adapted to produce face-pair-reduced, face-generic triangulations of (not necessarily prime) summands of $M.$
\end{rem}


\section{Squares and tetrahedra in face-generic triangulations}
\label{sec:squares}

The previous section focussed on the study of triangulations using faces and certain pairs of faces. We will now take the combinatorial study of triangulations further by examining twisted squares and pairs of twisted squares. This leads to a classification of the possible types of tetrahedra in a face-pair-reduced, face-generic triangulation.


\subsection{Squares in triangulations}
\label{subsec:Squares in triangulations}

Let $\sigma$ be an oriented tetrahedron in $\widetilde{\Delta}.$ A \emph{twisted square} in $\sigma$ is a properly embedded disc in $\sigma$ such that the boundary of the disc is the union of two pairs of opposite edges of $\sigma;$ see Figure~\ref{fig:quad+square}. Twisted squares first appeared in work of Thurston~\cite{thu}. Notice that each of the three quadrilateral types in $\sigma$ naturally corresponds to the twisted square, which the quadrilateral meets in its corners. The adjective \emph{twisted} will often be omitted.

\begin{figure}[h]
  \begin{center}
     \includegraphics[height=3cm]{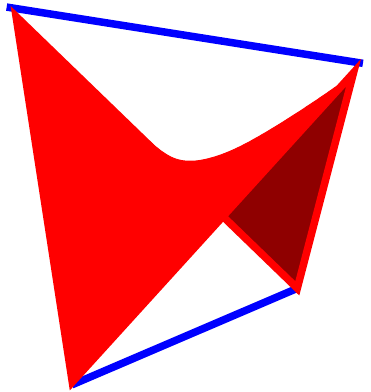}\qquad\qquad\qquad\qquad
     \includegraphics[height=3cm]{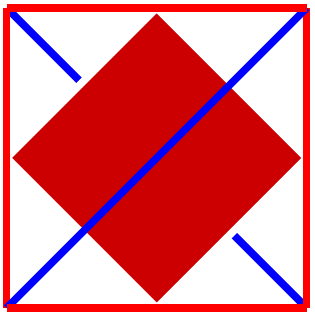}
            \end{center}
 \caption{Twisted square. On the left is a twisted square as realised by a fundamental piece of Schwarz's minimal surface, on the right is a quadrilateral disc and the boundary of its dual square.}
 \label{fig:quad+square}
\end{figure}

A twisted square has the following geometric realisation. Identify $\sigma$ with a regular Euclidean tetrahedron. Then the minimal disk spanning the union of two pairs of opposite edges of $\sigma$ is a twisted square. This is a fundamental piece of Schwarz's minimal surface. We will work with twisted squares combinatorially, but alternative arguments using the Gau\ss--Bonnet formula can be given because this realisation of a twisted square is isometric with an equilateral hyperbolic polygon with all dihedral angles $\pi/3.$

Under the map $\widetilde{\Delta}\to M,$ the edges of a square may become identified; the image may have 4, 3, 2 or 1 distinct edges in $M.$ Labelling the boundary of a square by letters indicating the \emph{unoriented} edges in $M$, we obtain the following basic \emph{partition-types of squares}:
$$
[A] \ abcd \quad [B] \ abac  \quad [C] \ aabc \quad [D] \ abab \quad [E] \ aaab \quad [F]  \ aabb \quad [G] \ aaaa,
$$
where different letters correspond to distinct (unoriented) edges in $M$.

\begin{figure}[h]
  \begin{center}
    \subfigure[A]{\includegraphics[height=2.3cm]{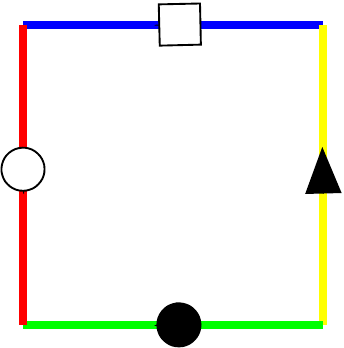}}
    \qquad
     \subfigure[B]{\includegraphics[height=2.3cm]{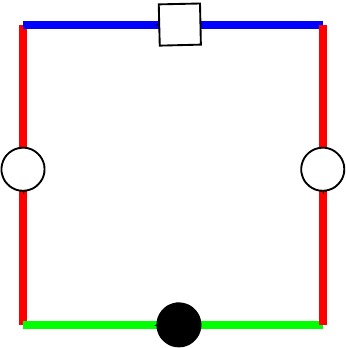}}
  \qquad
     \subfigure[C]{\includegraphics[height=2.3cm]{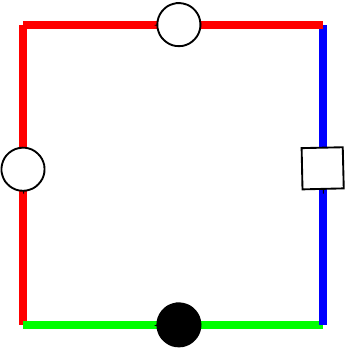}}
  \qquad
     \subfigure[D]{\includegraphics[height=2.3cm]{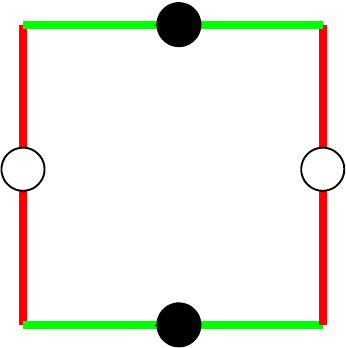}}
     \\
     \subfigure[E]{\includegraphics[height=2.3cm]{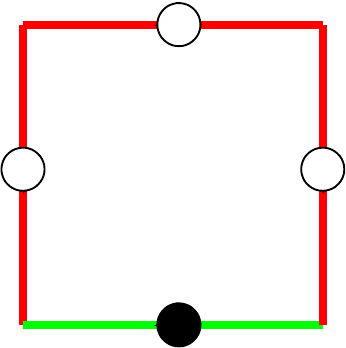}}
    \qquad
     \subfigure[F]{\includegraphics[height=2.3cm]{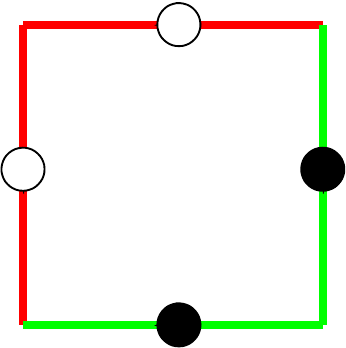}}
   \qquad
     \subfigure[G]{\includegraphics[height=2.3cm]{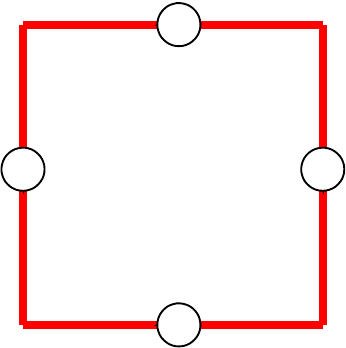}}
      \end{center}
 \caption{The partition types of squares (edges of the same colour are identified; edges of distinct colours are distinct)}
\end{figure}

It will be useful to have names for certain types of squares, as we did for triangles. We will only need names indicating the topology. A square of type $B$ can be an \emph{annulus square} or a \emph{M\"obius square}. A square of type $C$ is a \emph{M\"obius square}, a square of type $D$ is a \emph{torus square} or \emph{Klein square} or \emph{projective square}, a square of type $F$ is a \emph{Klein square}. See Figure~\ref{fig:square-topology}.

\begin{figure}[h]
  \begin{center}
    \subfigure[Annulus]{\includegraphics[height=2.3cm]{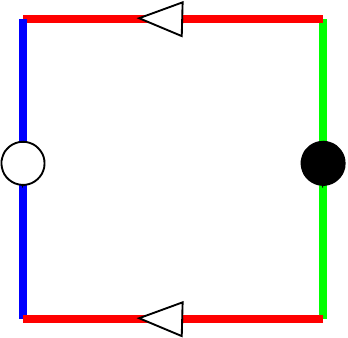}}
 \qquad
     \subfigure[M\"obius]{\includegraphics[height=2.3cm]{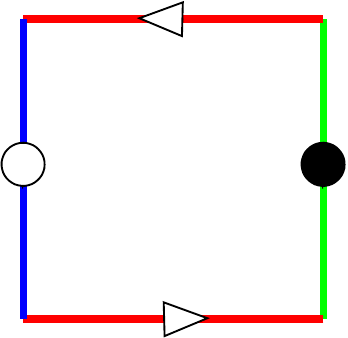}}
     \qquad
     \subfigure[M\"obius]{\includegraphics[height=2.3cm]{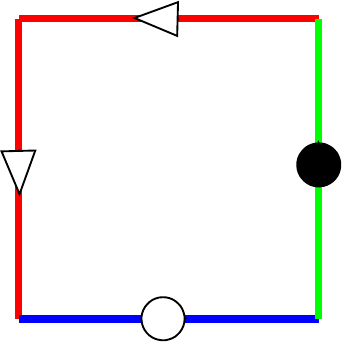}}
    \qquad
     \subfigure[Klein]{\includegraphics[height=2.3cm]{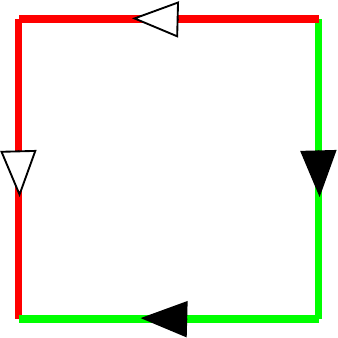}}\\
     \subfigure[Klein]{\includegraphics[height=2.3cm]{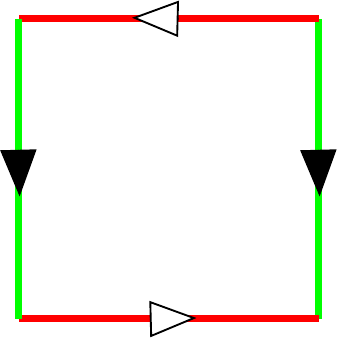}}
 \qquad
     \subfigure[Torus]{\includegraphics[height=2.3cm]{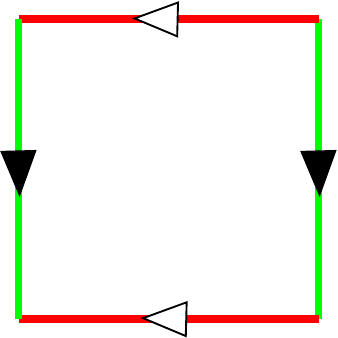}}
      \qquad
     \subfigure[Projective]{\includegraphics[height=2.3cm]{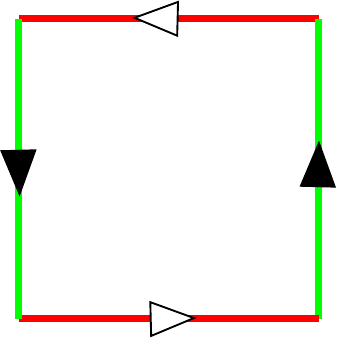}}
          \end{center}
   \caption{Some topological types of squares in a face generic triangulation}
   \label{fig:square-topology}
\end{figure}

We now define what we mean by a \emph{pinched surface} (see also Figure~\ref{fig:pinched surface}). Suppose $S$ is a compact surface and partition $S$ into finite subsets with the property that all but finitely many of these sets are singletons. Then the quotient space is a \emph{pinched surface}. Equivalently, consider a finite collection of pairwise disjoint subsets $V_k$ of $S.$ Then identifying all points in $V_k$ to one point for each $k$ gives a pinched surface.

\begin{figure}[h]
\begin{center}
\includegraphics[height=3.3cm]{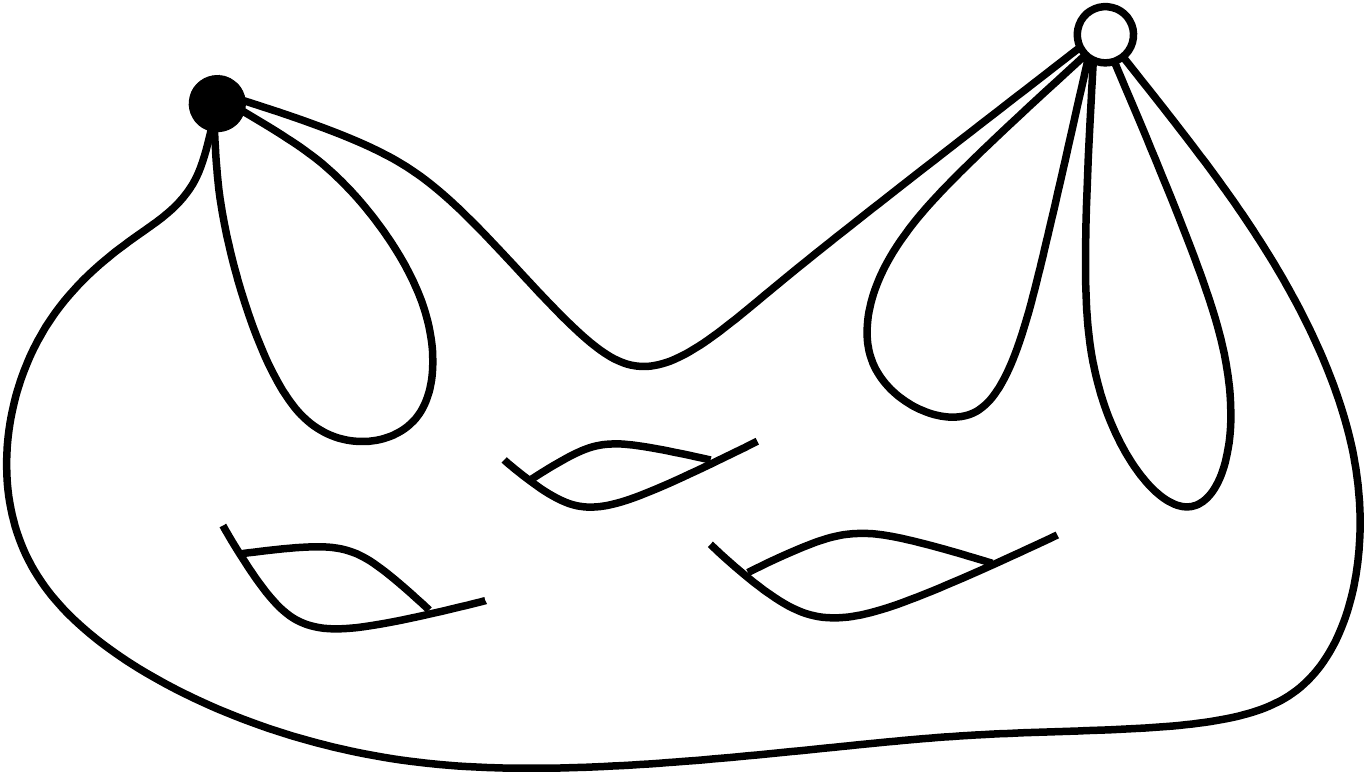}
      \end{center}
 \caption{A pinched surface}
 \label{fig:pinched surface}
\end{figure}

\begin{lem}\label{lem:G-F-D square lemma} 
Suppose $M$ is closed, orientable and has a face-generic triangulation. 
\begin{enumerate}
\item No square is of type $G.$
\item If there is a square of type $F,$ then $M$ contains an embedded Klein bottle.
\item If there is a square of type $D,$ and this square is not an embedded torus in $M,$ then $M$ either contains an embedded projective plane or an embedded Klein bottle.
\item If the triangulation is also face-pair-reduced, then no square of type $D$ is a projective square.
\end{enumerate}
\end{lem}

\begin{proof}
Since the triangulation is face-generic, no face is a cone, 3--fold or dunce hat. If two consecutive edges of a square are identified, then the face containing them is a M\"obius face. So if some square is of type $G,$ then some tetrahedron has four M\"obius faces. This completes the first case. The second also follows from the induced orientations of the edges.
The third follows from the observation that the possible orientations on the edges of a square of type $D$ either give an embedded torus, and embedded Klein bottle or a pinched projective plane. Intersecting the pinched projective plane with the boundary of a small neighbourhood of the vertex gives two disjoint simple closed curves, and capping these off with disjoint discs in the vertex neighbourhood gives an embedded projective plane in $M.$

The fourth claim follows by drawing in the remaining edges of a tetrahedron containing a square of type $D.$ The two faces are forced to be identified due to Definition~\ref{defn:face-pair-reduced}. However, the induced face pairing will be orientation preserving, contradicting the hypothesis that $M$ is orientable.
\end{proof}


\subsection{Pinched, capped and surgery 2--square surfaces}
\label{subset:pinched 2--square surface}

In this section, we describe certain unions of two twisted squares in distinct tetrahedra which allow us to draw conclusions about the topology of a triangulated 3--manifold.

\begin{lem}\label{lem:capped 2-square surface}
Let $M$ be a closed, orientable, triangulated 3--manifold, and $S$ be a closed, connected surface which has a CW decomposition, whose set of 2--cells consists of at most two squares. Then $-1 \le \chi(S) \le 2.$
Suppose further that $f\co S \to M$ is a combinatorial map that takes the squares to twisted squares in distinct tetrahedra, and distinct edges in $S$ to distinct edges in $M.$ Then there is an embedding $g\co S \to M.$
\end{lem}

\begin{proof}
Since there are four edges, two faces and at least one vertex, we have $-1 \le \chi(S) \le 2.$ In fact, $S$ can be viewed as obtained from a hexagon by identifying boundary edges in pairs. Now $f\co S \to M$ is a combinatorial map that takes the squares to squares in distinct tetrahedra, and distinct edges to distinct edges. Then $f(S)$ is a (possibly) pinched surface in $M.$ We can replace $f(S)$ by an embedded surface homeomorphic with $S$ as follows. The surface $f(S)$ is embedded outside the vertex neighbourhoods. The intersection with sufficiently small vertex linking spheres consists of a finite number of simple closed loops, which we may cap off using discs (starting from innermost circles) to obtain an embedded surface homeomorphic with $S.$ This surgery of $f$ gives the desired embedding $g\co S \to M.$
\end{proof}

\begin{figure}
  \begin{center}
    \subfigure[$P\# P\# P,$ $(C,C)$]{\includegraphics[height=2.1cm]{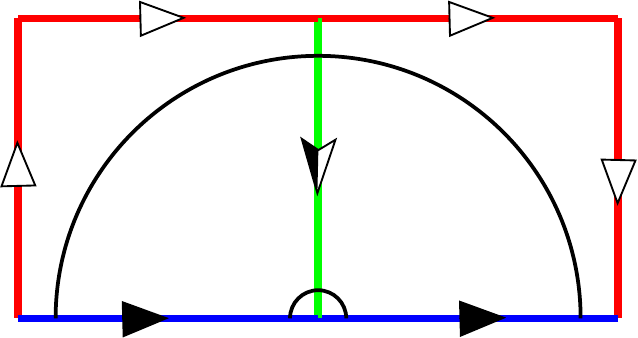}\label{fig:immersed squares-a}}
    \qquad
     \subfigure[$P\# P,$ $(C,C),$ $L(4)$]{\includegraphics[height=2.1cm]{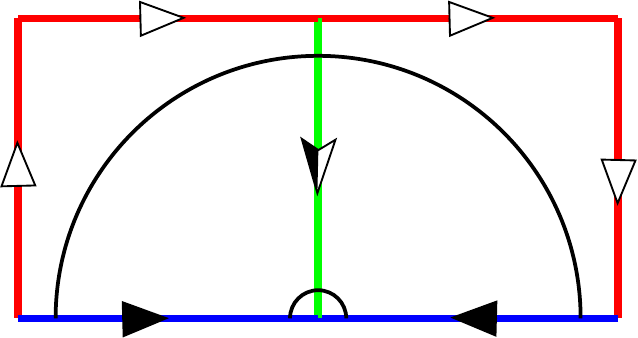}\label{fig:immersed squares-b}}
  \qquad
     \subfigure[$P,$ $(C,C),$ $L(2)$]{\includegraphics[height=2.1cm]{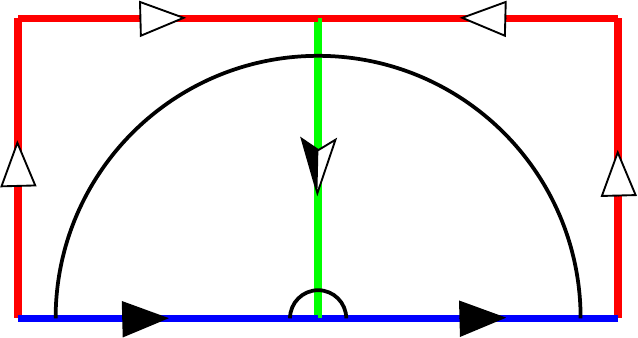}\label{fig:immersed squares-c}}
\\
     \subfigure[$P\# P,$ $(C,C)$]{\label{fig:immersed squares d}\includegraphics[height=2.1cm]{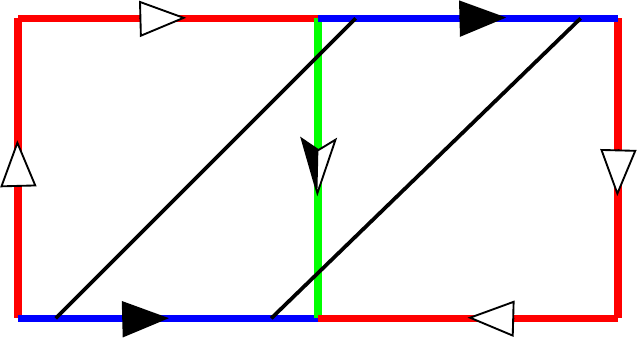}\label{fig:immersed squares-d}}
     \qquad
     \subfigure[$P\# P\# P,$ $(C,E)$]{\includegraphics[height=2.1cm]{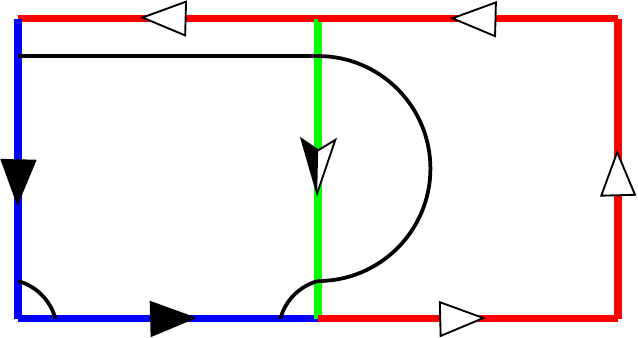}\label{fig:immersed squares-e}}
    \qquad
     \subfigure[$P\# P,$ $(C,E)$]{\includegraphics[height=2.1cm]{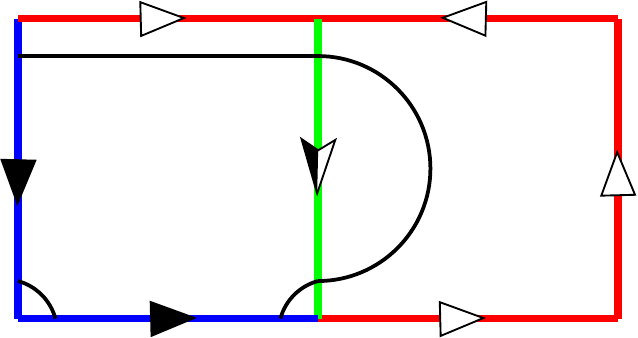}\label{fig:immersed squares-f}}
\\
      \subfigure[$P,$ $(B,B),$ $L(2)$]{\includegraphics[height=2.1cm]{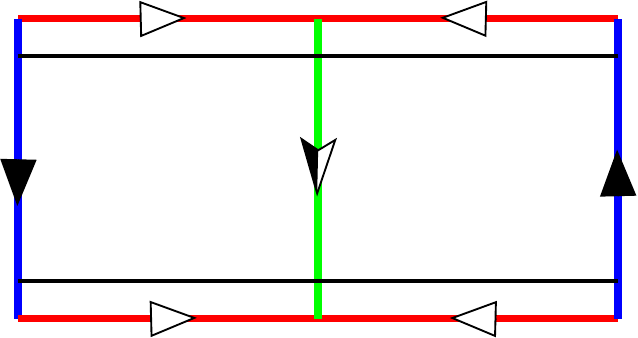}\label{fig:immersed squares-g}}
    \qquad
     \subfigure[$P\# P\# P,$ $(B,B)$]{\includegraphics[height=2.1cm]{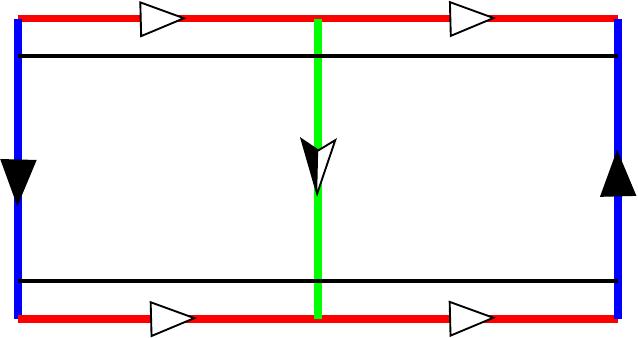}\label{fig:immersed squares-h}}
  \qquad
     \subfigure[$P\# P,$ $(B,B)$]{\label{fig:immersed squares i}\includegraphics[height=2.1cm]{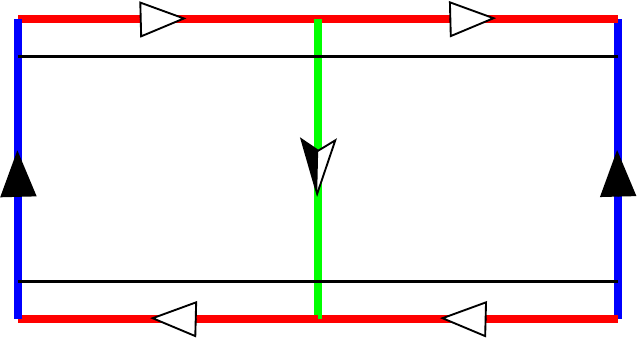}\label{fig:immersed squares-i}}
\\
     \subfigure[$P,$ $(B,B),$ $L(2)$]{\includegraphics[height=2.1cm]{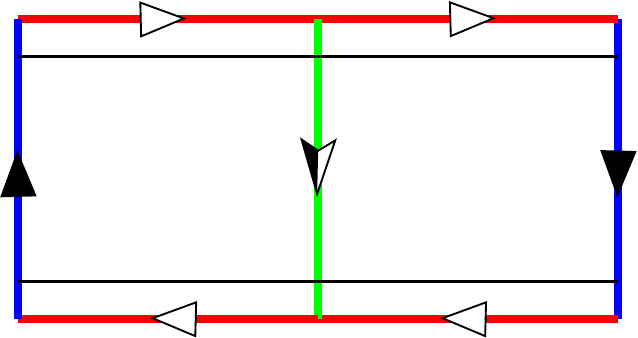}\label{fig:immersed squares-j}}
     \qquad
     \subfigure[$P,$ $(B,B),$ $L(2)$]{\includegraphics[height=2.1cm]{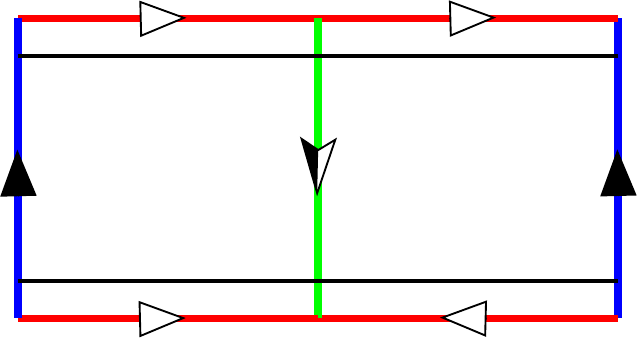}\label{fig:immersed squares-k}}
    \qquad
     \subfigure[$P\# P,$ $(B,B)$]{\includegraphics[height=2.1cm]{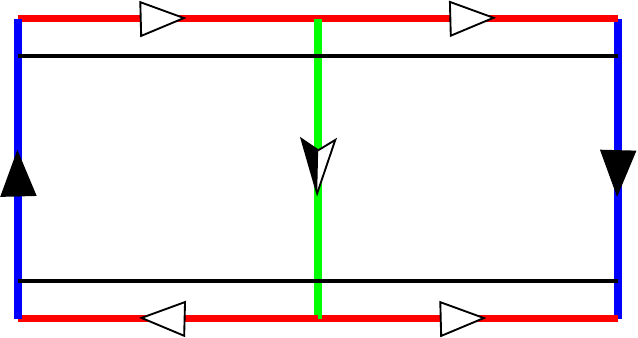}\label{fig:immersed squares-l}}    
\\
     \subfigure[$P\# P\# P,$ $(E,E),$ $L(6)$]{\includegraphics[height=2.1cm]{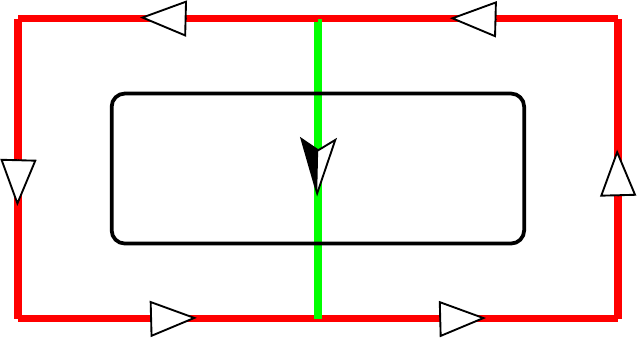}\label{fig:immersed squares-m}}
      \end{center}
 \caption{Pairs of squares in distinct tetrahedra detecting non-orientable surfaces or lens spaces (or lens space summands in case the manifold is reducible). Edges of distinct colours are distinct. The Klein bottles in \ref{fig:immersed squares d} and \ref{fig:immersed squares i} consist of an annulus in the edge neighbourhood and the annulus on the squares. Notice that some of these pictures are equivalent once one deletes the green diagonal from the hexagon.}
\label{fig:immersed squares}
\end{figure}

We call the map $f\co S \to M$ in the statement of Lemma~\ref{lem:capped 2-square surface} a \emph{pinched 2--square surface} in $M,$ and the embedding $g\co S \to M$ of the proof a \emph{capped 2--square surface}.

Now we take this analysis further in order to detect interesting surfaces in $M$ even if there are more edge identifications amongst two twisted squares. We are not trying to get a complete list, but rather list the situations that arise in this paper, and call the resulting surfaces \emph{surgery 2--square surfaces}. The basic fact we use is that if a simple closed curve on the boundary of a solid torus runs along the longitude $2n$ times, then it bounds a properly embedded surface of non-orientable genus $n$ with one boundary component in the solid torus. For instance, if $n=2$ then the surface is $(P\# P) \setminus \text{disc}$. We call such a curve a $2n$--curve on the solid torus, and we include the case $n=0$ to give a (boundary parallel or meridian) disc.

To illustrate the basic technique, suppose a square of type $G$ in $M$ has its boundary edge oriented coherently. Then the edge is a loop in $M,$ and intersecting the square with a regular tubular neighbourhood of that edge in $M$ gives a 4--curve on the neighbourhood, and hence bounds a $(P\# P) \setminus (\text{disc})$ in the tubular neighbourhood. Moreover, we can cap this off with the complement in the square of the neighbourhood and hence get an embedded Klein bottle in $M.$ One can view the square of type $G$ as the image of an immersion of a 1--square Klein bottle, and the surgery procedure has recovered an embedding. In practice, reference to a domain is not required, as the object given to us is the union of two twisted squares in $M.$

Some constructions with two squares are shown in Figure~\ref{fig:immersed squares}, and they should be interpreted as follows. Let $N$ denote a small regular neighbourhood of the red edge.
Suppose $X \subset M$ is the union of two twisted squares contained in distinct tetrahedra. Then $X$ meets $\partial N$ in the shown black curve(s). The intersection $N \cap X$ is replaced by an embedded surface in $N,$ which has the same boundary curve(s). In almost all cases, $X \cap \partial N$ is connected, and (after choosing an orientation) its class in $H_1(N)$ can be read off from the picture, taking orientations into account. In some cases, $X \cap \partial N$ has one component which is zero in $H_1(N)$ (and is hence capped off with a disc) and another which may or may not be trivial. In Figures~\ref{fig:immersed squares d} and \ref{fig:immersed squares i}, one has two parallel non-trivial curves and these are connected by an annulus in $N.$ In this case, an annulus on $X$ gives an orientation reversing isotopy between these curves; whence the surface is a Klein bottle. We will call a surface that arises as described above a \emph{surgery 2--square surface} in $M.$ 


\subsection{Recognising projective planes using three squares}

We continue the analysis of the previous section with three situations that arise at the end of this paper. Suppose there are three pairwise distinct tetrahedra $\sigma_k$ containing squares $S_k < \sigma_k$ such that the edges are identified as shown in Figure~\ref{fig:square-triple}. In the left-most picture, the subcomplex $X$ formed by these squares has an edge $e$ of degree 6 (shown in red). The intersection of the boundary of a small regular neighbourhood $N$ of $e$ with $X$ results in a curve, which is embedded in $\partial N$ and is null-homotopic in $N.$ It hence bounds a disc in $N.$ This curve also bounds a M\"obius band in $X,$ whence $M$ contains an embedded projective plane. A similar surgery argument results in a projective plane in the middle picture, whilst in the last picture, merely capping is required to yield such a surface.

\begin{figure}[h]
  \begin{center}
     \subfigure[Surgery]{\includegraphics[height=4cm]{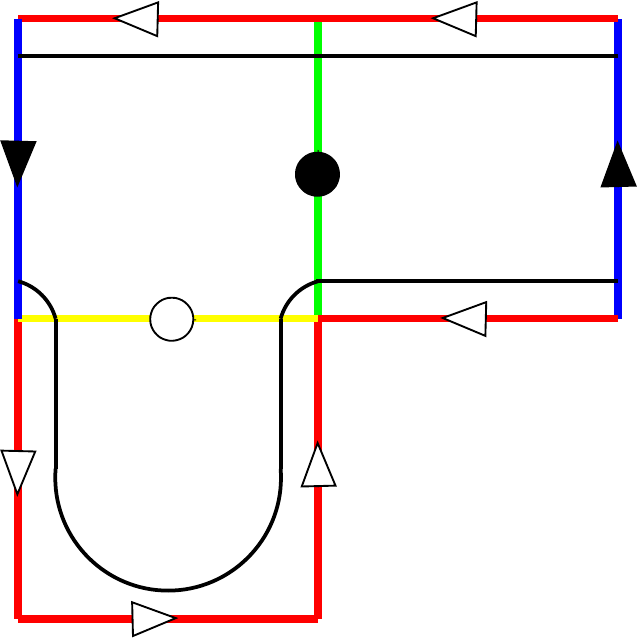} \label{fig:square-triple-a}}\qquad
     \subfigure[Surgery]{\includegraphics[height=4cm]{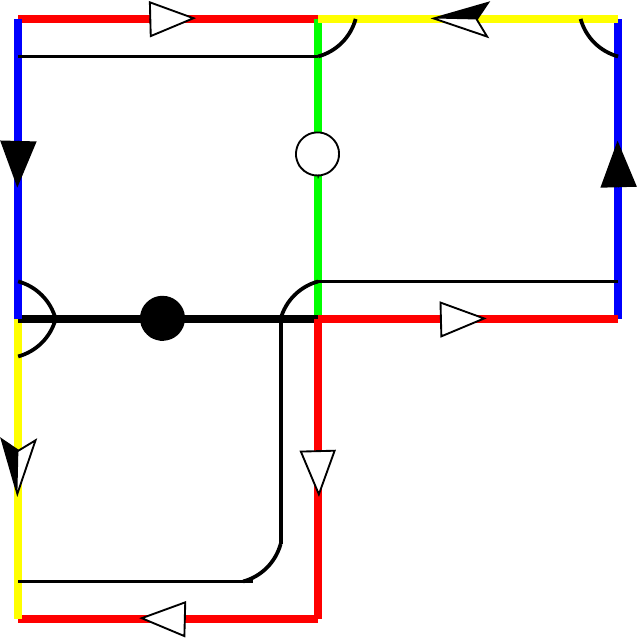} \label{fig:square-triple-b}}\qquad
     \subfigure[Capped]{\includegraphics[height=4cm]{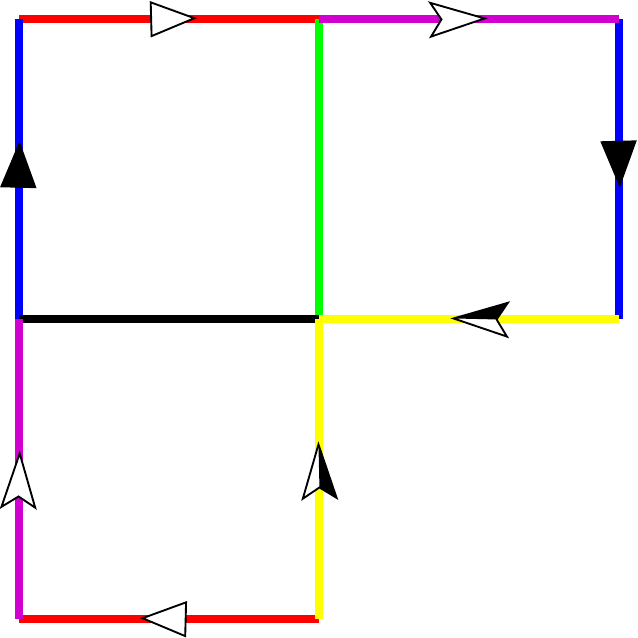} \label{fig:square-triple-c}}
            \end{center}
 \caption{Three squares contained in three pairwise distinct tetrahedra giving a projective plane}
 \label{fig:square-triple}
\end{figure}


\subsection{Partition-types of tetrahedra}

We now describe the combinatorics of a tetrahedron in a face-generic 1--vertex triangulation using the three squares contained in it. The tetrahedron $\sigma$ in $M$ is said to be of \emph{partition-type XYZ} if the three squares supported by it are of partition-types $X,$ $Y$ and $Z$ respectively. The labels are ordered lexicographically.

\begin{lem}\label{lem:edge identifications of tet}
Suppose $M$ is closed, orientable and has a face-generic triangulation. Suppose $\sigma$ is a tetrahedron in $M$, and let $n$ be the number of pairwise distinct edges of $\sigma.$ Then either
\begin{enumerate}
\item $n=6$ and $\sigma$ is of partition-type $AAA;$ or
\item $n=5$ and $\sigma$ is of partition-type $AAC$ or $ABB;$ or
\item $n=4$ and $\sigma$ is of partition-type $AAF,$ $ABE,$ $ACC,$ $BBC$ or $BBD;$ or
\item $n=3$ and $\sigma$ is of partition-type $BBF,$ $BDE$ or $DDD.$
\end{enumerate}
Moreover, in case $DDD,$ all three squares in $\sigma$ are embedded Klein bottles in $M,$ unless one square gives a pinched projective plane and hence and embedded projective plane in $M.$
\end{lem}

\begin{figure}[h]
  \begin{center}
    \subfigure[AAA]{\includegraphics[height=2.6cm]{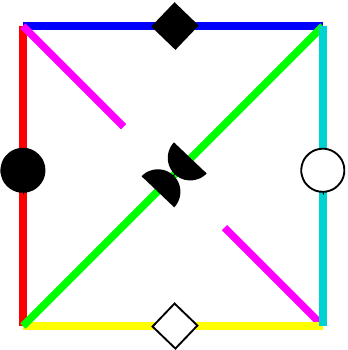}}
    \qquad
     \subfigure[AAC]{\includegraphics[height=2.6cm]{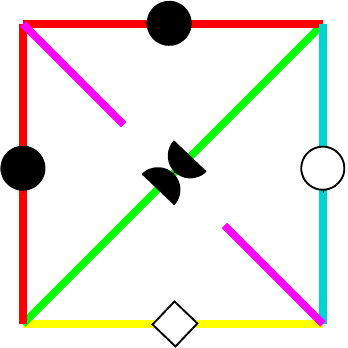}}
  \qquad
     \subfigure[ABB]{\includegraphics[height=2.6cm]{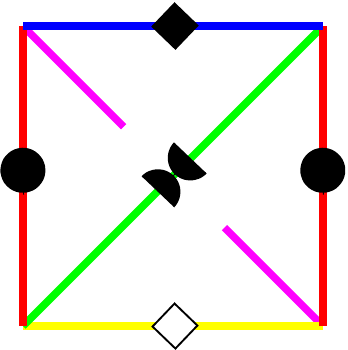}}
  \qquad
     \subfigure[AAF]{\includegraphics[height=2.6cm]{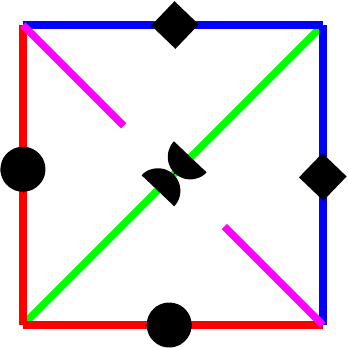}}
     \\
     \subfigure[ABE]{\includegraphics[height=2.6cm]{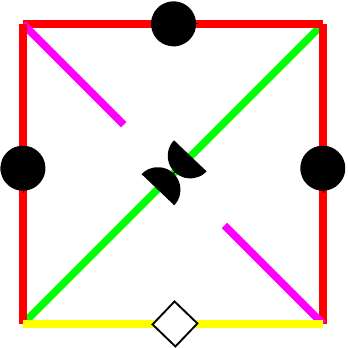}}
    \qquad
     \subfigure[ACC]{\includegraphics[height=2.6cm]{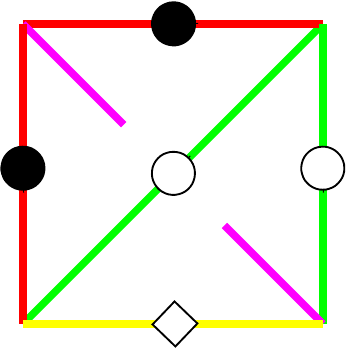}}
   \qquad
     \subfigure[BBC]{\includegraphics[height=2.6cm]{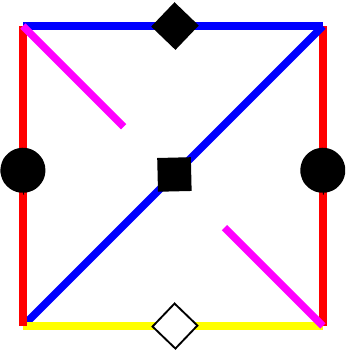}}
      \qquad
     \subfigure[BBD]{\includegraphics[height=2.6cm]{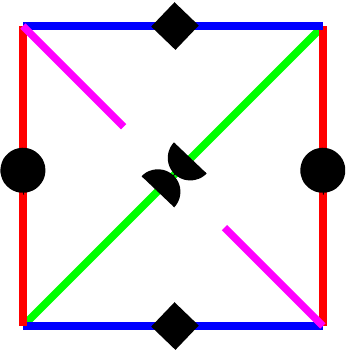}}
     \\
     \subfigure[BBF]{\includegraphics[height=2.6cm]{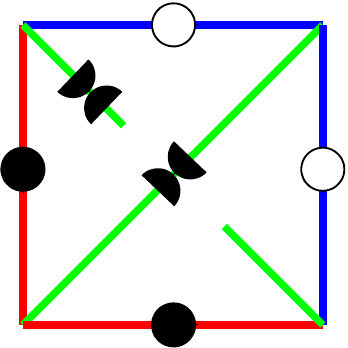}}
    \qquad
     \subfigure[BDE]{\includegraphics[height=2.6cm]{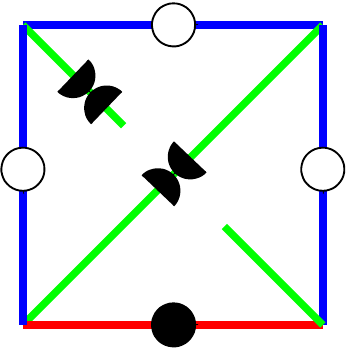}}
   \qquad
     \subfigure[DDD]{\includegraphics[height=2.6cm]{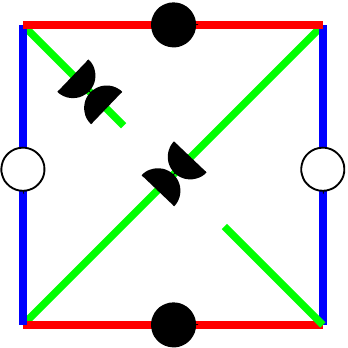}}
      \end{center}
 \caption{The partition types of tetrahedra (edges of the same colour are identified; edges of distinct colours are distinct; faces may or may not be identified)}
\end{figure}

\begin{proof}
Recall that face-generic implies that all faces are M\"obius faces or triangles, and $\sigma$ has at most two M\"obius faces. In particular, no square can be of type $G.$ 

Next suppose that some square is of type $F.$ Since no face is a cone, the two remaining edges of $\sigma$ are distinct from the edges of the square. If these two edges are identified, we have type $BBF$ and $n=3,$ and if they are distinct, we have type $AAF$ and $n=4.$ 

Hence suppose that no square is of type $F$ or $G.$ We enumerate the remaining cases depending on $n.$ First suppose $n=6.$ Then each square must be of partition type $A.$

Suppose $n=5.$ It follows that either a pair of opposite edges is identified, giving case $ABB,$ or a pair of edges incident with a face is identified, giving $AAC.$

Suppose $n=4.$ First assume that three edges are identified. Three given edges are either incident with one vertex, or contained in a unique face or in a unique square. The three edges cannot be incident with one vertex, since then there would be a cone face. Also, the three edges of a face cannot be identified since otherwise there would be a 3--fold face or a dunce face. Hence one of the squares is of type $E.$ Since this square contains exactly two edges of the triangulation of $M,$ the remaining two edges must be distinct and distinct from the edges in $E,$ giving $ABE.$ Next assume that two pairs of edges are identified. If two pairs of opposite edges are identified, we obtain $BBD.$ If one pair of opposite edges is identified and one pair of adjacent edges, the we obtain $BBC.$ If no pair of opposite edges is identified (and hence two pairs of adjacent edges), we obtain $ACC.$

Suppose $n=3.$ No square can be of type $A.$ Suppose there is a square of type $B.$ Then it contains the three distinct edges of $\sigma.$ Enumerating the possible identifications of the remaining two edges with the three edges of the $B$--square gives $BDE,$ $BBF,$ $BBG,$ $BCE.$  We have already assume that no square is of type $F$ or $G,$ and in case $BCE,$ there would be 3 M\"obius faces. Hence the only possibility is $BDE.$
Now then suppose no square is of type $B,$ but there is a square of type $C.$ In each possible case, one finds three or four M\"obius faces. Hence suppose no squares are of type $B$ or $C.$ If there is a square of type $D,$ then one of the remaining edges must be different, and since we don't have a type $B$ square, the remaining edges are identified. This gives $DDD.$ Now one of the squares gives a projective plane unless all three squares are Klein bottles. Last, suppose all squares are of type $E.$ This is not possible.

Suppose $n=2.$ In this case, we can only have squares of type $D$ or $E.$ The only possibility is $DEE,$ which gives four M\"obius faces.
\end{proof}

\begin{cor}\label{cor:tet types}
Suppose $M$ is closed, orientable and has a face-generic triangulation without Klein or projective squares. Then each tetrahedron is combinatorially equivalent to one of eight possibilities; one each falling into the partition types AAA, AAC, ABB, ABE, ACC, BBC, sBBD, sBDE as shown in Figure~\ref{fig:special-tet-types}. In the first six cases, this is just the generic type (up to combinatorial equivalence), but in the latter two, there are prescribed edge identifications. 

If, in addition, the triangulation is face-pair-reduced, then 
\begin{enumerate}
\item case $BBC$ does not occur, and 
\item the two M\"obius faces of every tetrahedron of type $sBDE$ in $M$ are identified (i.e.\thinspace such a tetrahedron is a so-called 1-tetrahedron layered solid torus $\LST(3,2,1)$).
\end{enumerate}
\end{cor}

\begin{figure}[h]
  \begin{center}
    \subfigure[AAA]{\includegraphics[height=2.5cm]{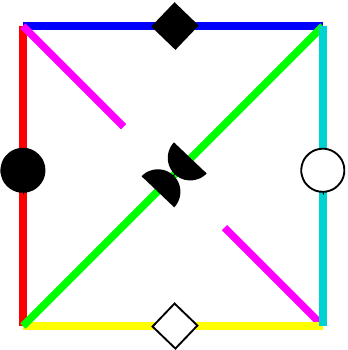}\label{fig:special-tet-types-AAA}}
    \qquad   \qquad
     \subfigure[AAC]{\includegraphics[height=2.5cm]{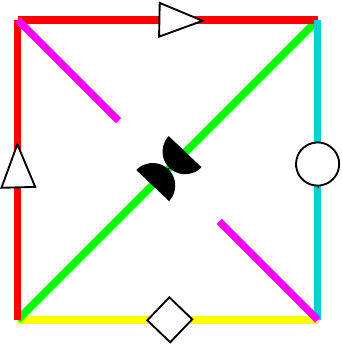}\label{fig:special-tet-types-AAC}}
  \qquad   \qquad
     \subfigure[ABB]{\includegraphics[height=2.5cm]{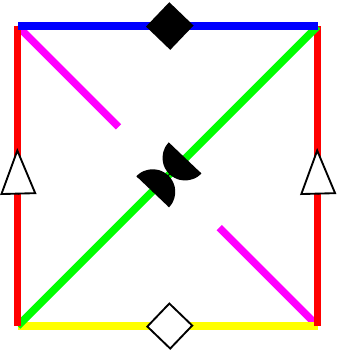}\label{fig:special-tet-types-ABB}}
      \qquad   \qquad
     \subfigure[ABE]{\includegraphics[height=2.5cm]{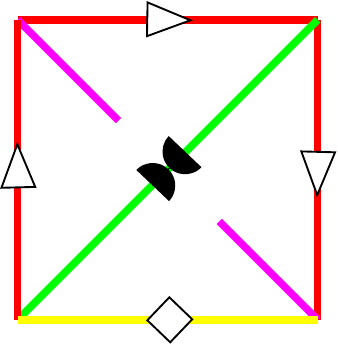}\label{fig:special-tet-types-ABE}}
    \\
     \subfigure[ACC]{\includegraphics[height=2.5cm]{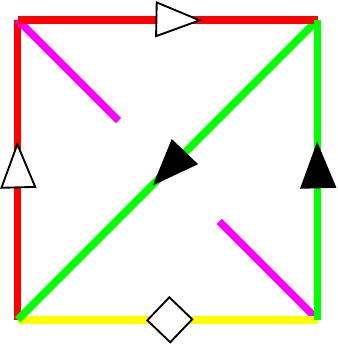}\label{fig:special-tet-types-ACC}}
   \qquad   \qquad
     \subfigure[BBC]{\includegraphics[height=2.5cm]{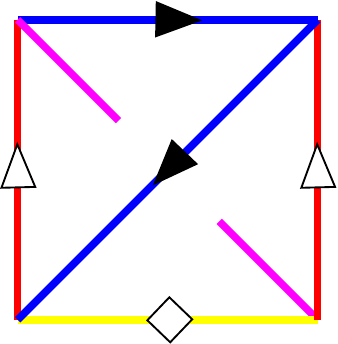}\label{fig:special-tet-types-BBC}}
      \qquad   \qquad
     \subfigure[sBBD]{\includegraphics[height=2.5cm]{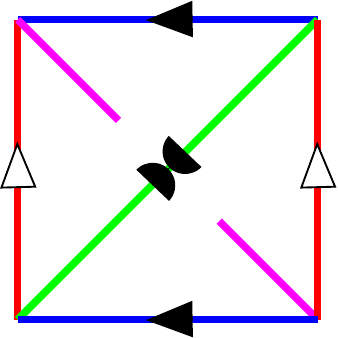}\label{fig:special-tet-types-BBD}}
      \qquad   \qquad
     \subfigure[sBDE]{\includegraphics[height=2.5cm]{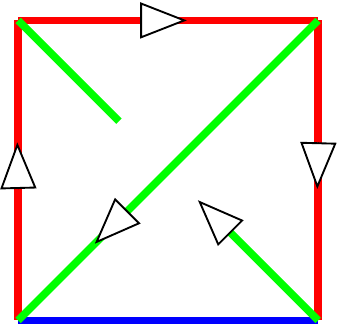}\label{fig:special-tet-types-BDE}}
       \end{center}
 \caption{The remaining possibilities for tetrahedra up to combinatorial equivalence (edges of the same colour are identified; edges of distinct colours are distinct)}
 \label{fig:special-tet-types}
\end{figure}

\begin{proof}
Refer to the types listed in Lemma~\ref{lem:edge identifications of tet}. In each of the cases $AAF,$ $BBF,$ and $DDD,$ there is a Klein or projective square, and hence they cannot occur.

Since we assume there are no Klein or projective squares, the squares of type $D$ in $BBD$ and $BDE$ give tori.   This determines $sBBD,$ and also, together with the hypothesis that there are no cone faces, determines $sBDE.$
\end{proof}



\section{Extremal rays with support in one tetrahedron}
\label{sec:Extremal rays with support in one tetrahedron}

There is an algebraic approach to normal surface theory due to Haken (see \cite{JR}), which arises from the fact that the normal discs in the cell structure of a normal surface match up across faces in the triangulation. Another algebraic approach (which goes back to Thurston and Jaco and can be found in Tollefson~\cite{tollefson}), relies on the key observation that each normal surface is uniquely determined (up to normal isotopy) by the quadrilateral discs in its cell structure, and that the normal surface meets a small neighbourhood of each edge of the triangulation in a disjoint union of discs. Intuitively, these equations arise from the fact that as one circumnavigates the earth, one crosses the equator from north to south as often as one crosses it from south to north. In terms of the abstract neighbourhood of an edge, one picks an orientation of the edge (so that north and south are defined) and observes that triangles meeting the edge remain in one of the hemispheres, whilst quadrilaterals either cross from north to south (in which case the corner of the quadrilateral at $e$ is given sign $-1$) or from south to north (in which case the corner of the quadrilateral at $e$ is given sign $+1$). See Figure~\ref{fig:slopes}. 

\begin{figure}[h]
    \centering
    \subfigure[The abstract neighbourhood]{%
        \label{fig:matchingquadbdry}%
        \includegraphics[height=4cm]{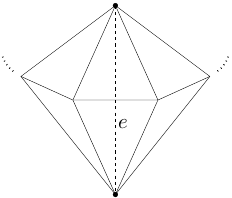}}
    \qquad
    \subfigure[Positive slope $+1$]{%
        \label{fig:matchingquadpos}%
        \includegraphics[height=4cm]{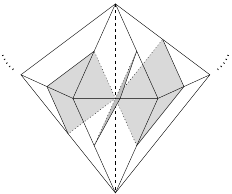}}
    \qquad
    \subfigure[Negative slope $-1$]{%
        \label{fig:matchingquadneg}%
        \includegraphics[height=4cm]{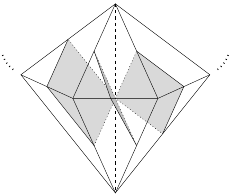}}
    \caption{Slopes of quadrilaterals}
    \label{fig:slopes}
\end{figure}

Given the edge $e$ and a quadrilateral type $q,$ the total weight $\text{wt}_e(q)$ of $q$ at $e$ is the sum of signs over all corners of $q$ at $e$ (where the empty sum has value $0$). For the quadrilateral type $q$ introduce the variable $x(q).$
The $Q$--matching equation at $e$ is then defined as $0 = \sum_q \text{wt}_e(q) x(q),$ where the sum is taken over all quadrilateral types. For instance, placing a quadrilateral of each type in a fixed tetrahedron $\sigma$ gives a so-called \emph{tetrahedral solution} to these equations: this has $x(q)=1$ for each quadrilateral type $q$ supported by $\sigma$ and $x(q)=0$ for all others. See \cite{tollefson, tillus-normal} for more details.

We call a solution to the matching or $Q$--matching equations a \emph{1--quad type solution} if it has exactly one non-zero quadrilateral coordinate, and a \emph{2--quad type solution} if it has exactly two non-zero quadrilateral coordinates. The non-zero quadrilateral coordinates need not have the same sign. If $x$ is \emph{any} solution to the matching or $Q$--matching equations with support in exactly one tetrahedron, then one may add a multiple of a tetrahedral solution to $x$ so that the result has at most two quadrilateral coordinates that are non-zero and so that both are non-negative.

\begin{figure}[h]
  \begin{center}
    \subfigure[A]{\includegraphics[height=2.5cm]{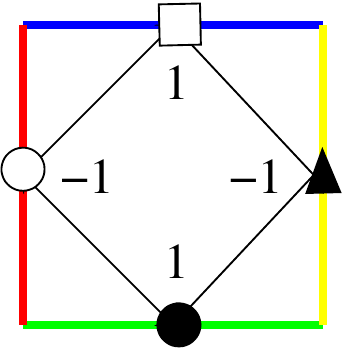}}
    \qquad \qquad
     \subfigure[B]{\includegraphics[height=2.5cm]{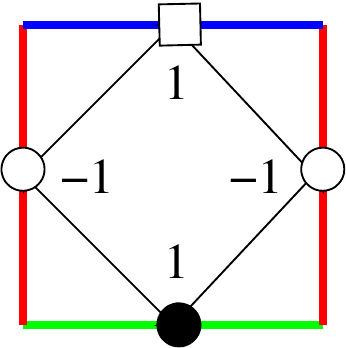}}
\qquad \qquad
     \subfigure[B]{\includegraphics[height=2.5cm]{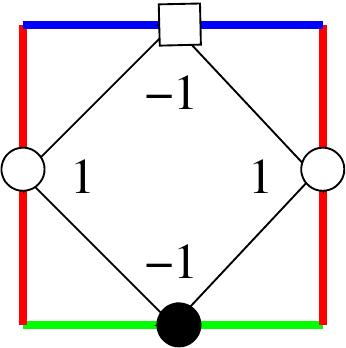}}
  \\
     \subfigure[C]{\includegraphics[height=2.5cm]{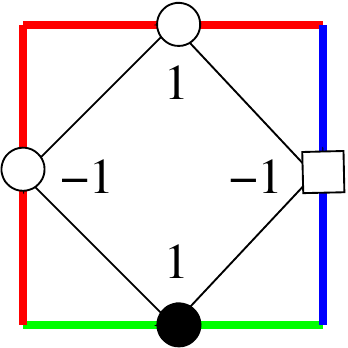}}
  \qquad \qquad
     \subfigure[D]{\includegraphics[height=2.5cm]{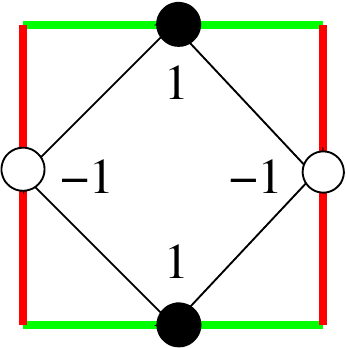}}
     \qquad \qquad
     \subfigure[E]{\includegraphics[height=2.5cm]{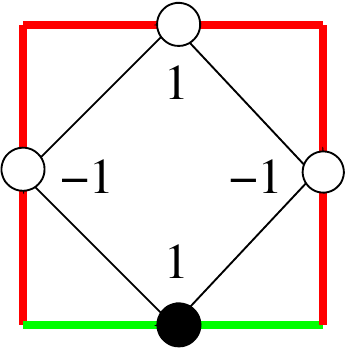}}
    \\\
    \subfigure[E]{\includegraphics[height=2.5cm]{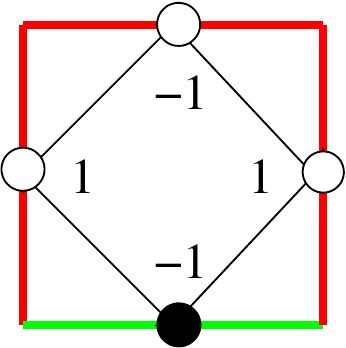}}
    \qquad \qquad
     \subfigure[F]{\includegraphics[height=2.5cm]{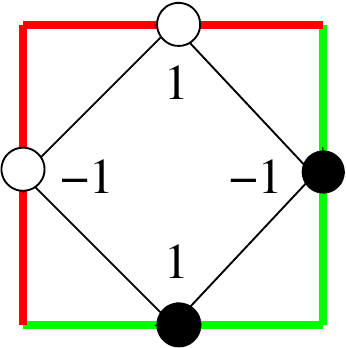}}
   \qquad \qquad
     \subfigure[G]{\includegraphics[height=2.5cm]{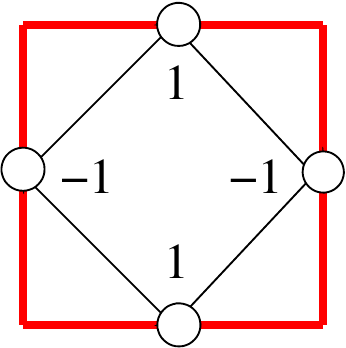}}
      \end{center}
 \caption{The signs of corners of quadrilateral discs shown with the dual twisted squares (edges of the same colour are identified; edges of distinct colours are distinct)}
\label{fig:partition types of quads}
\end{figure}

\begin{lem}\label{lem:1-quad type}
Let $M$ be a closed, orientable 3--manifold with an arbitrary triangulation. The tetrahedron $\sigma$ in $M$ supports a 1--quad type solution if and only if at least one square in $\sigma$ is of type $F$ or $G.$ 

If the triangulation is face-generic, then $M$ supports a 1--quad type solution if and only if it contains a Klein square.

If the triangulation is 0--efficient and no face is a cone, then $M$ supports a 1--quad type solution if and only if it either contains a Klein square or a square is a spine for $L(4,1).$
\end{lem}

\begin{proof}
In an orientable 3--manifold, the signs at the corners of a quadrilateral disc in a tetrahedron are the same at opposite corners and opposite at adjacent corners (see \cite{tillus-normal}). A quadrilateral disc has its four corners on the boundary of a twisted square; see Figure~\ref{fig:partition types of quads}. There is a 1--quad type solution if and only if the signs of the square at each edge sum to zero. It can now be deduced from the figure that there is a 1--quad type solution if and only if there is a twisted square of type $F$ or $G.$
This proves the first assertion. 

The second assertion follows as in the proof of Lemma~\ref{lem:G-F-D square lemma}.

Hence assume that the triangulation is 0--efficient and no face is a cone. Then a square of type $F$ is a Klein square and a square of type $G$ is a spine for $L(4,1).$ 
\end{proof}

The lemma implies that the existence of a 1--quad type solution is easy to deduce from the combinatorics of a triangulation.

\begin{pro}\label{pro:2-quad type in 1 tet}
Suppose a triangulation of the closed, orientable 3--manifold $M$ supports no 1--quad type solution. 

If the triangulation is face-generic, then there is no 2--quad type solution supported by precisely one tetrahedron $\sigma.$

If the triangulation is 0--efficient and there is a 2--quad type solution supported by precisely one tetrahedron $\sigma$, then either a face of $\sigma$ is a cone (in which case $M$ is homeomorphic with $S^3$) or $\sigma$ is of type $DEE$ (in which case $M$ is homeomorphic with $L(5)$). 
\end{pro}

\begin{proof}
Suppose the 2--quad type solution $x$ does not vanish on the two quadrilateral coordinates $p, q < \sigma.$ If $x(p)= x(q),$ then one can adjust by a tetrahedral solution to obtain a 1--quad type solution which does not vanish precisely on the remaining quadrilateral coordinate supported by $\sigma,$ contradicting the assumption that there is no 1--quad type solution.

After possibly adjusting by a tetrahedral solution and relabelling, we may assume that $x(p)>x(q)>0.$ Let $\alpha=x(p)$ and $\beta=x(q).$ Since the space of all solutions to the $Q$--matching equations is a rational polyhedral cone and there is no 1--quad type solution, it follows that $\alpha$ and $\beta$ are linearly dependent over $\Q.$ This fact, however, is irrelevant.

Up to orientation, we may assume that the weights on the six edges of $\sigma$ are $\alpha, \alpha, \beta-\alpha, \beta-\alpha, -\beta, -\beta.$ Since there are precisely two positive weights, and we assume that there is no 1--quad type solution, it follows that the two edges of $\sigma$ with weight $\alpha$ are distinct in $M.$ Since $\alpha\neq \beta,$ is follows that each edge of positive weight is identified with at least two edges of negative weight. Whence with precisely two edges of negative weight. Up to symmetry and ignoring orientation of edges, this gives precisely two possible patterns for the identification of the edges of $\sigma.$ In one case, there are three edges incident with a vertex of $\sigma,$ which must be identified. But then a face is a cone, giving a contradiction when the triangulation is face-generic, and the conclusion that $M$ is $S^3$ when the triangulation is 0--efficient.

In the other case, $\sigma$ is of type $DEE$ and a square and a face of $\sigma$ form a spine for $L(5).$ However, there is no tetrahedron of type $DEE$ in a face-generic triangulation according to Corollary~\ref{cor:tet types}.
\end{proof}

\begin{cor} \label{cor:face-generic+no Klein squares => no cluster in one tet}
If there is a 1--quad type or 2--quad type solution supported by a single tetrahedron in a 0--efficient triangulation, then the underlying manifold is $S^3$, $L(5)$ or contains an embedded Klein bottle. Moreover, a face-generic triangulation without Klein squares does not admit such a solution.
\end{cor}

As an example, the Seifert fibered space $S^2(\;(2,1),\; (2,1),\; (3,-1)\;)$ has a minimal 4--tetrahedron triangulations with a 1-quad type solution; this is triangulation \#1 of this manifold in Burton's closed, orientable manifold census \cite{Regina}. This triangulation is 0--efficient, face-generic and contains a Klein square.

\begin{cor}
Let $M$ be a closed, orientable 3--manifold with face-generic triangulation $\tri.$ 
If there are no Klein squares, then each tetrahedral solution lies on an extremal ray in $Q(\tri),$ the polyhedral cone of all non-negative solutions to the $Q$--matching equations.
\end{cor}



\section{Clusters of three 2--quad type solutions}
\label{sec:Clusters of three 2--quad type solutions}

We are now in a position to state, and prove, the main application of our study of face-pair-reduced, face-generic triangulations:

\begin{thm}\label{thm:cluster}
Suppose $M$ is a closed, orientable 3--manifold with face-pair-reduced, face-generic triangulation such that 
\begin{enumerate}
\item there are no Klein squares,
\item no two squares are identified as shown in Figure~\ref{fig:immersed squares},
\item no three squares are identified as shown in Figure~\ref{fig:square-triple},
\item no capped 2--square surface is a non-separating torus,
\item no capped 2--square surface is non-orientable.
\end{enumerate}
Then there is no cluster of three 2--quad type solutions. 
\end{thm}

\begin{rem}
The two minimal triangulations of $\RR P^3$ have clusters of three 2--quad type solutions. However, one of these triangulations is face-generic and not face-pair-reduced, and the other is not face-generic.
\end{rem}

\begin{proof}[Proof of Theorem~\ref{thm:main} using Theorem~\ref{thm:cluster}]
Let $M$ be a closed, orientable 3--manifold. We may suppose that $M$ is irreducible and atoroidal (since otherwise we are in the first or second case) and that we are given a minimal triangulation. According to \S\S\ref{subsec:Faces in triangulations}--\ref{subsec:Pairs of faces in triangulations}, either the triangulation is face-pair-reduced and face-generic or $M$ is a Seifert fibered space (which yields the third case). 
Hence assume that the minimal triangulation is face-pair-reduced, face-generic and that we have a cluster of three 2-quad type solutions. Then one of the five hypotheses in Theorem~\ref{thm:cluster} does not hold. 

The negation of (1) implies that $M$ is toroidal or a Seifert fibered space.

The negation of (2) or (5) implies that $M$ contains an embedded non-orientable surface of Euler characteristic $-1,$ $0$ or $1.$ In the first case, there is an embedded non-orientable surface of genus 3, in the second $M$ is toroidal or Seifert fibered and in the third $M$ is real projective 3--space and hence Seifert fibered.

The negation of (3) implies that $M$ is real projective 3--space.

The negation of (4) is that $M$ contains a non-separating torus. This is homologically non-trivial. Since $M$ is irreducible, every sphere is homologically trivial, and hence the torus must be incompressible, and hence $M$ is toroidal.
This concludes the proof of Theorem~\ref{thm:main}.
\end{proof}

\begin{proof}[Proof of Theorem~\ref{thm:cluster}]
For the remainder of this section, suppose the hypotheses of Theorem~\ref{thm:cluster} are satisfied and that the triangulation has at least 3 tetrahedra, but that there is a cluster of three 2--quad type solutions. Suppose that $M$ is oriented and that all tetrahedra are oriented coherently. Denote the solutions $x_1, x_2, x_3$ and suppose $\sigma_0$ is a tetrahedron in their common support. We denote the pairwise distinct quadrilateral types $q_1, q_2, q_3<\sigma_0$ such that $x_i(q_i)\neq 0.$ For each solution $x_i,$ there also is tetrahedron $\sigma_i$ and quadrilateral $p_i < \sigma_i$ such that $x_i(p_i)\neq 0.$ The tetrahedra $\sigma_1, \sigma_2, \sigma_3$ are not necessarily pairwise distinct. However, since the triangulation is face-generic and has no Klein squares, it follows from Corollary~\ref{cor:face-generic+no Klein squares => no cluster in one tet} that $\sigma_1, \sigma_2$ and  $\sigma_3$ are distinct from $\sigma_0.$

\emph{Outline:} A classification of the possible types of tetrahedra in a face-pair-reduced, face-generic triangulation was given in Corollary~\ref{cor:tet types}. We will first focus on one 2--quad type solution $x$ supported on quadrilateral types $q$ and $p$ and determine the possible partition-types for the squares associated with the quadrilateral types. With this information, we then show that $\sigma_0$ has no square of type $B.$ We finally show that $\sigma_0$ has no square of type $A,$ hence contradicting the classification of tetrahedra in Corollary~\ref{cor:tet types}. 

\emph{Orientation:} In some arguments, we will only analyse certain situations up to combinatorial equivalence or up to orientation in order to limit the number of cases to be considered. However, this has to be taken into account when more than one of these situations are put together into a common framework since $M$ is oriented and all tetrahedra are oriented coherently.


\emph{Partition-types of quadrilaterals:}
The $Q$--matching equations and signs of corners of quadrilateral discs have been reviewed in \S\ref{sec:Extremal rays with support in one tetrahedron}. Let $\sigma$ be an oriented tetrahedron in $M,$ containing the quadrilateral disc $q.$ Then $q$ has signs $+1, +1, -1 , -1$ at the edges of a twisted square in $\sigma.$ In $M,$ these signs could be at 4, 3, 2 or 1 distinct edges as shown in Figure~\ref{fig:partition types of quads}. We list these cases according to the partition type of the square containing the corners of the quadrilateral, and refer to this as the \emph{partition type} of the quadrilateral.

The notation we use is explained with an example. The partition $(1, -1\mid1\mid-1)$ indicates that the signs are at three distinct edges in $M$ and the respective weights are $1$ and $-1$ at one edge, $1$ at the second edge and $-1$ at the third edge. Lemma~\ref{lem:1-quad type} implies that no square is of type $F$ or $G$ since the triangulation is face generic and has no Klein square, so we have the following cases:
\begin{enumerate}
\item[A] $(1\mid1\mid-1\mid-1)$
\item[B] $(1,1 \mid -1 \mid -1 )$  or $(-1,-1\mid1\mid1)$                                                  
\item[C] $(1, -1\mid1\mid-1)$
\item[D] $(-1,-1\mid1,1)$
\item[E] $(1,1,-1\mid-1)$ or $(-1,-1,1\mid1)$
\end{enumerate}

Let $x$ be a 2--quad type solution to the $Q$--matching equations with support in two distinct tetrahedra, $\sigma_a$ and $\sigma_b.$ After possibly scaling and relabelling the tetrahedra, we may assume that $x(q)=1$ and $x(p)=t,$ where $q<\sigma_a,$ $p<\sigma_b$ and $0 < |t| \le 1.$ To simplify notation, we let $s = |t|.$ Note that according to Lemma~\ref{lem:1-quad type} and Proposition~\ref{pro:2-quad type in 1 tet}, neither $\sigma_a$ nor $\sigma_b$ supports a 1--quad type or 2--quad type solution, and $q<\sigma_a$ and $p<\sigma_b$ are of partition-type $A$, $B$, $C$, $D$ or $E.$ If $q$ is of partition-type $X$ and $p$ is of partition-type $Y,$ we will say that $x$ is a partition-type $(X,Y)$ solution. We will now determine the possible partition-types $(X,Y),$ as well as the corresponding weights $(1, s).$ 

We first consider the possible weights on the edges of $\sigma_b.$ It follows from the $Q$--matching equations, that the sum of all weights at each edge has to equal zero. Since the partition type takes the number of pairwise distinct edges and their weights into account, in order to determine which types can possibly balance other types, one needs to determine how many edges have non-zero weight. This is given by the \emph{reduced partition types of quadrilaterals}, where only non-zero weight sums are counted:
\begin{enumerate}
\item[$A'$] $(1\mid1\mid-1\mid-1)$
\item[$B'$] $(2\mid-1\mid-1)$  or $(-2\mid1\mid1)$                                                  
\item[$C'$] $(1\mid-1)$
\item[$D'$] $(-2\mid2)$
\item[$E'$] $(1\mid-1)$ 
\end{enumerate}
Since no further identifications are allowed between the edges of a given square, it follows that the only potential partition-types $(X, Y)$ for the solution $x$ are 
\begin{enumerate}
\item[] $(X,X),$ where $X \in \{ A, B, C, D, E\},$ or
\item[] $(X, Y),$ where $X, Y \in \{C, D, E\}$ and $X\neq Y.$
\end{enumerate}

We now narrow down the profile of \emph{some} of the partition-types using the five hypotheses of Theorem~\ref{thm:cluster}. Our strategy is to consider the twisted squares $S_q \subset \sigma_a$ and $S_p \subset \sigma_b$ dual to $q<\sigma_a$ and $p<\sigma_b$ respectively. 


\emph{Claim 0: \ For each of the partition types $(A,A),$ $(B,B),$ $(C,C),$ $(C,D),$ $(D,C)$ there is (up to combinatorial equivalence) a unique identification between the edges of $S_q$ and $S_p$  as shown in Figure~\ref{fig:companion-types}. Moreover, the type $(C, E)$ does not occur.}

\begin{figure}[h]
  \begin{center}
      \subfigure[(A, A)]{\includegraphics[height=2.9cm]{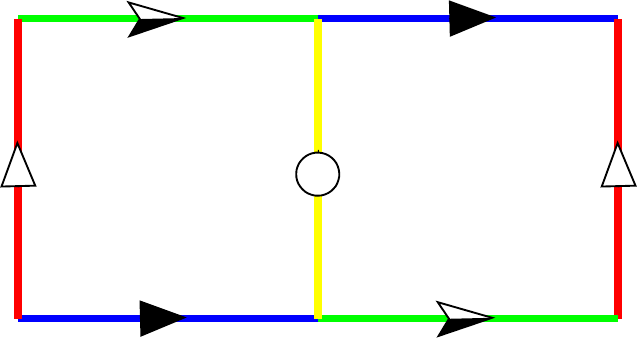}\label{fig:companion-types-a}}
    \qquad   \qquad
    \subfigure[(B, B)]{\label{fig:companion-types BB}\includegraphics[height=2.9cm]{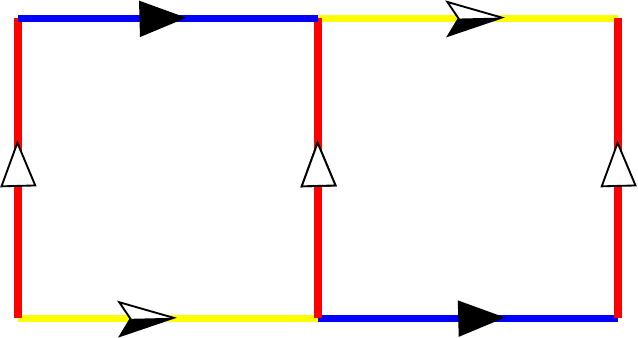}\label{fig:companion-types-b}}
 \\
     \subfigure[(C, C)]{\includegraphics[height=2.9cm]{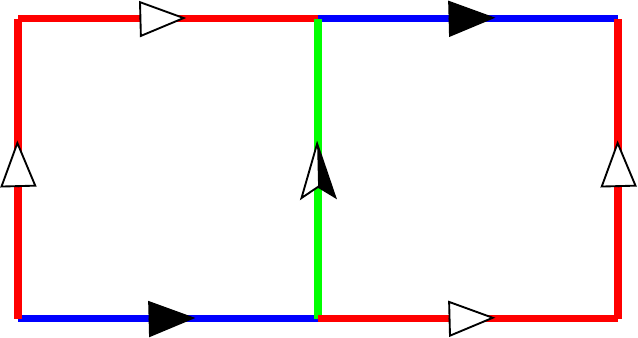}\label{fig:companion-types-c}}
   \qquad   \qquad
 \subfigure[(C, D)]{\includegraphics[height=2.9cm]{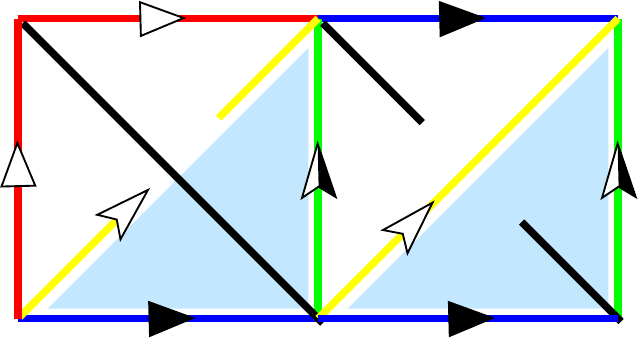}\label{fig:companion-types-d}}
          \end{center}
 \caption{Shown tetrahedra are coherently oriented in each pair (but not necessarily between different pairs); edges of distinct colours are distinct; except the black edges could be any colour (restrictions apply) and do not need to be of the same colour; the shaded faces are identified.}
 \label{fig:companion-types}
\end{figure}

We will prove the claim one partition type at a time:

$(A, A)$: Each of the squares $S_q$ and $S_p$ has four pairwise distinct edges in $M.$ The signs of the corners of $q$ and $p$ sum to zero at each edge in $M,$ and hence each edge in $S_q$ is identified with a unique edge of $S_p$ and the signs at these edges are opposite. In particular, the union of $S_q$ and $S_p$ gives a capped 2-square surface $S.$ Due to hypothesis (5), $S$ must be orientable. Lemma~\ref{lem:capped 2-square surface} states that $-1\le \chi(S)\le 2,$ and hence $S$ is a torus or sphere. Since the signs of corners of $p$ and $q$ cancel in pairs and alternate around the squares, there are only two possible pictures, a hexagonal torus and a sphere. The hexagonal torus is shown in Figure~\ref{fig:companion-types-a}, and we will now show that the case of a sphere does not occur. 

In this case, the identifications of the edges of the squares are as shown in Figure~\ref{fig:companion-types-AA}. 
We may assume that $\sigma_a$ is the left-hand tetrahedron in each of Figure~\ref{fig:companion-types-a-s-1} and \ref{fig:companion-types-a-s-2} and oriented as shown, and there are two possibilities for the orientation of $\sigma_b.$ Definition~\ref{defn:face-pair-reduced} implies that the two faces shaded in Figure~\ref{fig:companion-types-a-s-1} must be identified, which is not possible in an oriented 3--manifold. 
Applying Definition~\ref{defn:face-pair-reduced} to Figure~\ref{fig:companion-types-a-s-2} using the red and green edges forces the two shaded faces to be identified. The remaining "back faces" are similarly identified using the blue and yellow edges. The same reasoning applies to the front faces. Thus, the triangulation of $M$ consists of exactly two tetrahedra (it is in fact a triangulation of $S^3$ with 4 vertices and 6 edges), contradicting the fact that we have at least 3 tetrahedra. 

\begin{figure}[h]
  \begin{center}
      \subfigure[]{\includegraphics[height=2.9cm]{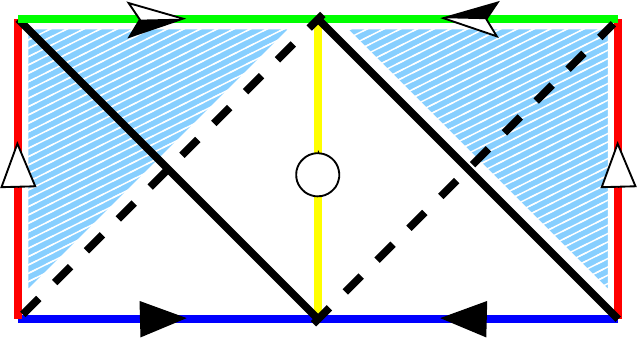}\label{fig:companion-types-a-s-1}}
    \qquad   \qquad
    \subfigure[]{\label{fig:companion-types BB}\includegraphics[height=2.9cm]{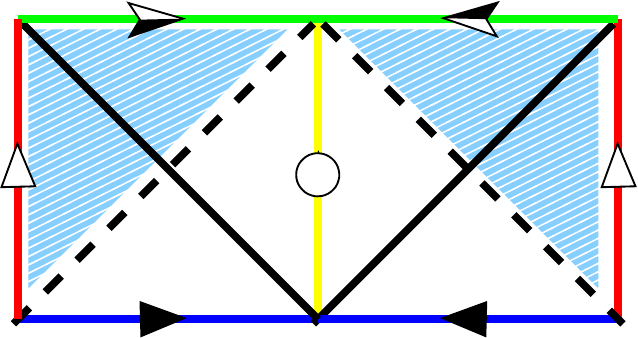}\label{fig:companion-types-a-s-2}}
          \end{center}
 \caption{Shown tetrahedra for $(A,A)$ are coherently oriented in each pair; edges of distinct colours are distinct, except the black edges could be any colour (restrictions apply) and do not need to be of the same colour.}
 \label{fig:companion-types-AA}
\end{figure}

$(B, B)$: If each of the squares $S_q$ and $S_p$ is of type $B,$ then their reduced partition types must be $(2\mid-1\mid-1)$ and $(-2\mid1\mid1),$ with the edges of weights $2$ and $-2$ identified. A square of type $B$ is either an  annulus or a M\"obius square. This gives three main cases. 

Suppose that one square is a M\"obius square and the other is an annulus square. Up to combinatorial equivalence, this gives two possibilities, shown in Figures~\ref{fig:immersed squares-k} and \ref{fig:immersed squares-l}.

Next, suppose that both squares are M\"obius squares. Up to combinatorial equivalence, this gives three possibilities. Two of these possibilities are shown in Figures~\ref{fig:immersed squares-i} and \ref{fig:immersed squares-j}. The remaining is shown in Figure~\ref{fig:companion-types-bb-mo}. One now argues as in case $(A,A)$: there are two possible cases depending on the orientations of $\sigma_a$ and $\sigma_b.$ Definition~\ref{defn:face-pair-reduced} again implies that one of these cannot occur due to the orientability of $M,$ and the other forces the four faces of $\sigma_a$ and $\sigma_b$ to be identified, giving a two tetrahedron triangulation. We already have the desired contradiction, but would like to point out that the four face pairings imply that all squares in the identification space are of type A, contradicting the fact that $S_q$ and $S_p$ are of type $B.$ So this situation cannot occur in a face-generic, face-pair reduced, orientable 3--manifold triangulation.

\begin{figure}[h]
  \begin{center}
    \subfigure[Two M\"obius squares]{\label{fig:companion-types BB}\includegraphics[height=2.9cm]{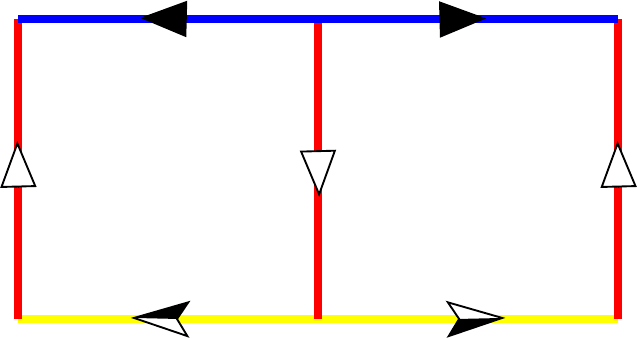}\label{fig:companion-types-bb-mo}}
     \qquad   \qquad   
     \subfigure[Two annulus squares]{\includegraphics[height=2.9cm]{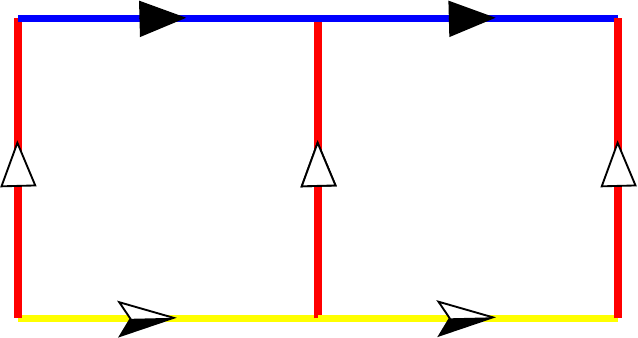}\label{fig:companion-types-bb-an}}
          \end{center}
 \caption{Shown are the squares for the remaining possibilities for $(B,B),$ which cannot occur in a face-generic, face-pair reduced, orientable 3--manifold triangulation.}
 \label{fig:companion-types-BB}
\end{figure}

Last, suppose both squares are annulus squares. Up to combinatorial equivalence, this gives four possibilities. Two of these possibilities are shown in Figures~\ref{fig:immersed squares-g} and \ref{fig:immersed squares-h}. The third is shown in Figure~\ref{fig:companion-types-bb-an} and, exactly as above, ruled out with Definition~\ref{defn:face-pair-reduced}. This only leaves the one possibility shown in Figure~\ref{fig:companion-types-b}.

$(C, C)$: Each square of type $C$ is topologically a M\"obius band, possibly with some vertices identified. Since a square of type $C$ has reduced partition type $(1\mid -1),$ each of the edges of degree one in one square is identified with a degree one edge of opposite sign of the other square. If this is the only identification between edges of the squares, then 
the union of $S_q$ and $S_p$ is a \emph{non-orientable} pinched 2-square surface (either a Klein bottle or $P^2\# P^2\# P^2$). We therefore have a non-orientable capped 2-square surface in $M,$ contradicting hypothesis (5).

The only other possibility is that, in addition, the degree two edges in $S_q$ and $S_p$ are also identified, so the 2-complex made up of $S_q$ and $S_p$ has one degree 4 edge and two degree 2 edges. Up to combinatorial equivalence, there are six possibilities. Four of these are shown in Figures~\ref{fig:immersed squares-a}--\ref{fig:immersed squares-d}. The fifth possibility is shown in Figure~\ref{fig:companion-types-cc}, and, as above, Definition~\ref{defn:face-pair-reduced} implies that this is not possible. 
The last possibility is shown in Figure~\ref{fig:companion-types-c}.

\begin{figure}[h]
  \begin{center}
  \includegraphics[height=2.9cm]{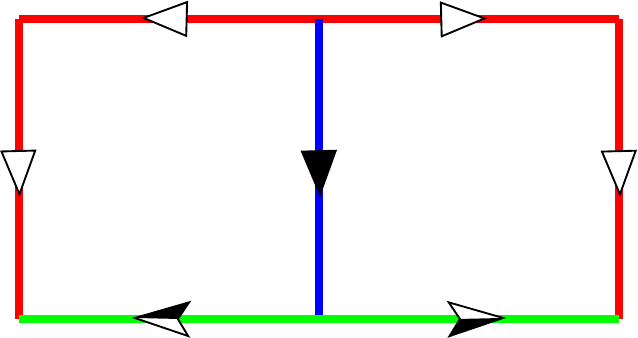}\label{fig:companion-types-cc-cont}
          \end{center}
 \caption{Shown are the squares for the fifth possibility for $(C,C),$ which cannot occur in a face-generic, face-pair reduced, orientable 3--manifold triangulation.}
 \label{fig:companion-types-cc}
\end{figure}

$(C, D)$ or $(D,C)$: Lemma~\ref{lem:G-F-D square lemma}  and hypothesis (1) imply that the square of type $D$ is a torus square. So up to combinatorial equivalence, there is only one possibility for the identification of the two squares. Moreover,  Definition~\ref{defn:face-pair-reduced} implies that two faces of $\sigma_a$ and $\sigma_b$ are identified and thus, up to orientation, there is only one possibility, shown in Figure~\ref{fig:companion-types-d}. The face identification yields an identification of two \emph{diagonals}. Definition~\ref{defn:face-pair-reduced} also implies that this diagonal cannot be identified with any of the edges in the square of type $D,$ and face-generic implies that it can also not be identified with the core edge of the M\"obius square of type $C.$

$(C, E)$: This case does not occur: the only two possibilities are shown in Figures~\ref{fig:immersed squares-e} and \ref{fig:immersed squares-f}, giving either a Klein bottle or $P^2\# P^2\# P^2.$


\emph{Claim 1: \ The tetrahedron $\sigma_0$ is of type $AAA$, $AAC$ or $ACC.$}

To prove the claim, suppose that at least one of the 2--quad type solutions is of type $B$ in $\sigma_0.$ According to \emph{Claim 0}, the only possibility for a companion for a type $B$ solution is type $B$ and as shown in Figure~\ref{fig:companion-types BB}. In particular, the squares of type $B$ must be annuli. Now we assume that there are three 2--quad type solutions, and their types are given by the type of $\sigma_0.$ Refer to the types of tetrahedra in Figure~\ref{fig:special-tet-types}. In ABB, BBC, sBBD, sBDE, one square of type $B$ is a M\"obius square. Hence $\sigma_0$ is of type $ABE.$

We know that a companion for a square of type $A$ (respectively $B$) is also of type $A$ (respectively $B$). Suppose that the companion square of type $A$ is supported by $\sigma_1,$ and the companion square of type $B$ is supported by $\sigma_2.$ 

If $\sigma_1 \neq \sigma_2,$ then (using the classification from \emph{Claim 0}), the square of type $A$ in $\sigma_1,$ the square of type $B$ in $\sigma_2$ and the square of type $E$ in $\sigma_0$ form a 2--complex that is combinatorially equivalent to the one shown in Figure~\ref{fig:square-triple-a}, contradicting the hypothesis of Theorem~\ref{thm:cluster}.

\begin{figure}[h]
  \begin{center}
      \subfigure[Shown are $\sigma_0,$ $\sigma_1$ and part of the vertex link for one possible orientation of $\sigma_1.$]{\includegraphics[height=3.7cm]{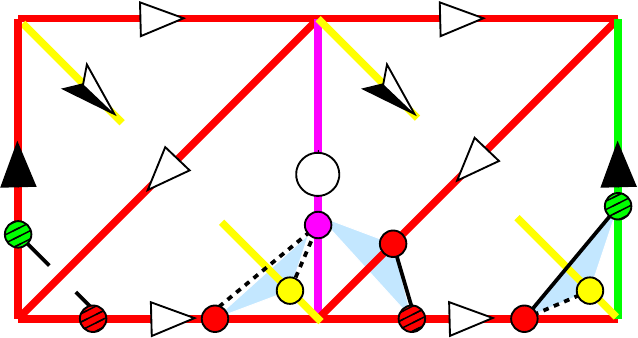}
       \qquad\qquad
          \includegraphics[height=3.3cm]{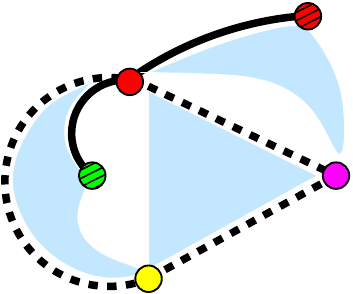}}
    \\
    \subfigure[The same as above for the other possible orientation of $\sigma_1.$]{\includegraphics[height=3.7cm]{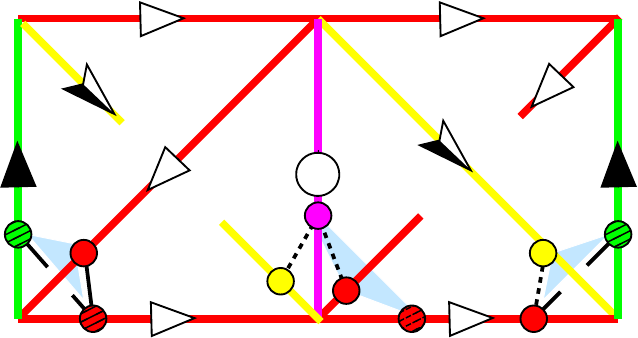}
          \qquad\qquad
          \includegraphics[height=3.3cm]{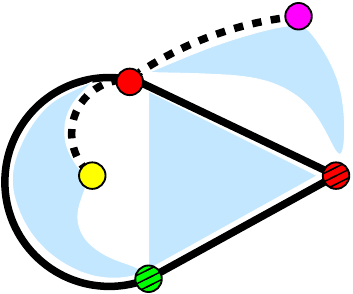}
    }
          \end{center}
 \caption{(ABE) The two possible orientations for the companion of a tetrahedron of type ABE. In each case, the indicated normal arcs form two simple closed curves that meet transversely in one point on the vertex link (the shaded triangles in the link are viewed from the vertex). } 
 \label{fig:ABEfinalcases}
\end{figure}

Hence $\sigma_1=\sigma_2.$ Without loss of generality, we may assume that the tetrahedron $\sigma_0$ of type $ABE$ is oriented as in Figure~\ref{fig:special-tet-types}. Since $\sigma_1$ contains squares of type $A$ and $B,$ this tetrahedron is either of type $ABE$ or of type $ABB,$ due to the classification in Figure~\ref{fig:special-tet-types}. However, if $\sigma_1$ is of type $ABB,$ then the two squares of type $B$ in $\sigma_1$ meet in a single edge $e$ in $M,$ and this edge is not contained in the square of type $A$ in $\sigma_1.$ However, in $\sigma_0$ this edge is in the intersection of the squares of types $A$ and $B.$ Whence $\sigma_1$ is of type $ABE.$ 

We need to distinguish two cases, depending on whether $\sigma_1$ is oriented as in Figure~\ref{fig:special-tet-types} or whether it is oppositely oriented. These cases are shown in Figure~\ref{fig:ABEfinalcases} (a) and (b) respectively. Examining the normal curves indicated in Figure~\ref{fig:ABEfinalcases} shows that the dashed arcs form a simple closed loop on the vertex link, and so do the solid arcs, but these two loops meet transversely in a single point. This contradicts the fact that the vertex link is a sphere. It follows that none of the 2--quad type solutions is of type $B$ in $\sigma_0,$ leaving only the three stated possibilities for the type of $\sigma_0.$


\emph{Claim 2: \ The tetrahedron $\sigma_0$ is not of type $ACC.$}

Suppose $\sigma_0$ is of type $ACC;$ we may assume that it is oriented as shown in Figure~\ref{fig:special-tet-types-ACC}. We know that a companion for a square of type $A$ is again of type $A.$ Supposing that the companion square of type $A$ is supported by $\sigma_1,$ there are two possible orientations for $\sigma_1.$ In either case, a M\"obius face of $\sigma_0$ has two distinct edges in common with a face of $\sigma_1.$ Since the triangulation is face-pair reduced, the faces must be identified. Since face pairings are orientation reversing, this only leaves one of the two possible orientations of $\sigma_1.$  Again applying the fact that the triangulation is face-pair reduced, one finds another pair of faces that is be identified. Examining the resulting vertex link shows that it cannot be a sphere; refer to Figure~\ref{fig:ACCfinalcase}.

\begin{figure}[h]
  \begin{center}
      \includegraphics[height=3.7cm]{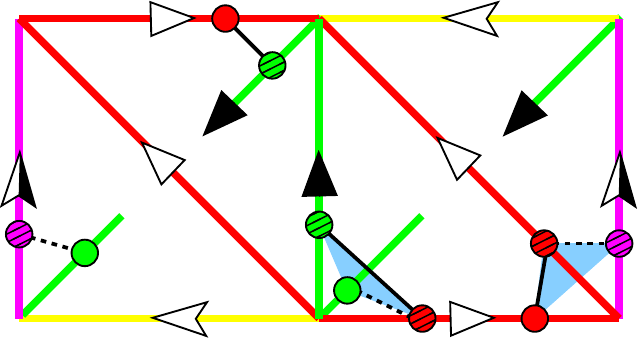}
      \qquad\qquad
          \includegraphics[height=3.3cm]{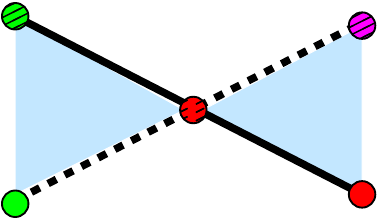}
          \end{center}
 \caption{Shown is the only possibility for the companion of a tetrahedron of type ACC. The indicated normal arcs form two simple closed curves that meet transversely in one point on the vertex link. This can be seen by viewing the shaded triangles from the vertex; as the two remaining normal arcs do not intersect in the vertex link, the latter cannot be a sphere.} 
 \label{fig:ACCfinalcase}
\end{figure}


\emph{Claim 3: \ The tetrahedron $\sigma_0$ is not of type $AAC.$}

Suppose $\sigma_0$ is of type $AAC;$ we may assume that it is oriented as shown in Figure~\ref{fig:special-tet-types-AAC}. We know that a companion for a square of type $C$ is of type $C$ or $D.$

First suppose that the companion square is also of type $C$ and supported by $\sigma_1.$ Also suppose that one of the squares of type $A$ has companion square supported by $\sigma_1.$ For either of the two choices of the square of type $A,$ this forces $\sigma_1$ to be of type $AAC$ also. 
Moreover, $\sigma_0$ and $\sigma_1$ then have M\"obius faces with the same core and boundary edges, and so these faces must be identified. This forces the orientation of $\sigma_1.$ 
Examining the signs of quadrilaterals, it follows that we may also suppose that the companion of the other square of type $A$ is supported by $\sigma_1.$ The result of this discussion is shown in Figure~\ref{fig:AAC_C comp C}. Examining the vertex linking triangles contradicts the fact that the vertex link is a sphere.

\begin{figure}[h]
  \begin{center}
      \includegraphics[height=3.7cm]{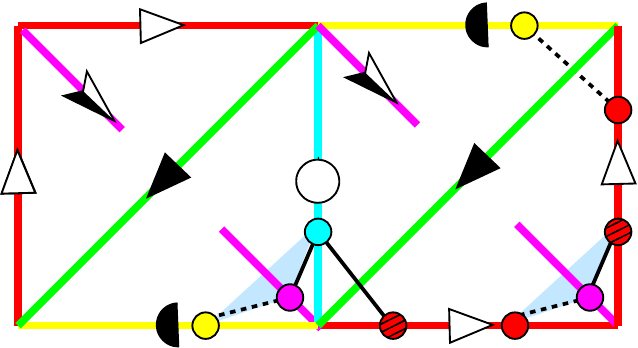}
      \qquad\qquad
         \includegraphics[height=3.3cm]{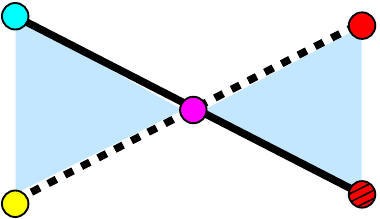}
          \end{center}
 \caption{Shown is the first possibility in case AAC. The indicated normal arcs form two simple closed curves that meet transversely in one point on the vertex link. The shaded triangles are viewed from the vertex; as the two additional normal arcs do not intersect in the vertex link, the latter cannot be a sphere.} 
 \label{fig:AAC_C comp C}
\end{figure}

Hence if the square of type $C$ has a companion square of type $C,$ then one of the squares of type $A$ has companion square supported by $\sigma_2\neq \sigma_1,$ and the other by $\sigma_3\neq \sigma_1.$ If $\sigma_2\neq \sigma_3,$ then the three companion squares are identified as shown in Figure~\ref{fig:square-triple-b}, contradicting our hypothesis. Hence $\sigma_2= \sigma_3.$ The 
union of the two squares of type $A$ in $\sigma_0$ is incident with exactly five pairwise distinct edges in $M$ and hence cannot have the same companion square in $\sigma_2.$ Moreover, by \emph{Claim 1}, $\sigma_2$ must be of type $AAC.$ Examining the signs of quadrilaterals, it follows that we may replace 
the solution $x_1$ by a solution $x'_1$ which is supported by squares of type $C$ in $\sigma_0$ and $\sigma_2.$ But this is again the same situation of the previous paragraph, where all three solutions are supported in two tetrahedra.

It follows that the companion square for $C$ is of type $D.$ This in particular implies that $\sigma_2 \neq \sigma_1$ and $\sigma_3\neq \sigma_1,$ since a tetrahedron containing a square of type $D$ has no square of type $A.$ If $\sigma_2 = \sigma_3,$ then as above, there is a solution $x'_1$ which is supported by squares of type $C$ in $\sigma_0$ and $\sigma_2,$ hence putting us into the above case. Whence $\sigma_2 \neq \sigma_3.$ The fact that the triangulation is face-pair reduced yields identifications between pairs of faces of the tetrahedra in such a way that not all of them can be orientation reversing; contradicting our hypothesis that the triangulation is oriented (see Figure~\ref{fig:AACfinalcase}). So AAC cannot happen.

\begin{figure}[h]
  \begin{center}
      \includegraphics[height=9cm]{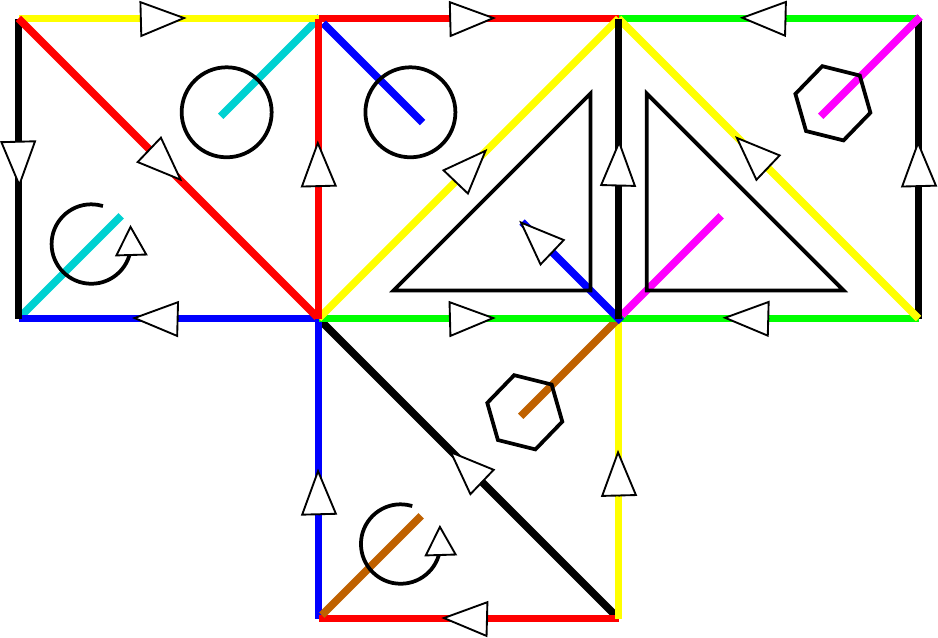}
          \end{center}
 \caption{Shown are $\sigma_0$ in the middle, with $\sigma_1$ on the right, $\sigma_3$ on the left and $\sigma_2$ below. The edge identifications force identifications of faces, and hence orientations of the tetrahedra. For instance, given an orientation of $\sigma_0,$ starting with the pair of faces labelled with a circle determines the orientation of $\sigma_3.$ The pair of faces labelled with a triangle determined the orientation of $\sigma_1.$ Then the faces labelled with hexagons determine the orientation of $\sigma_2.$ But then the faces labelled with a circular arc cannot be identified by an orientation reversing face pairing.} 
 \label{fig:AACfinalcase}
\end{figure}


\emph{Claim 4: \ The tetrahedron $\sigma_0$ is not of type $AAA.$}

First notice that again we may assume that either $\sigma_1=\sigma_2=\sigma_3,$ or else the three tetrahedra may be assumed to be pairwise distinct. In case they are pairwise distinct, the three companion squares give the capped projective plane shown in Figure~\ref{fig:square-triple-c}.

Hence $\sigma_1=\sigma_2=\sigma_3.$ There are two cases to distinguish depending on the orientation. The first case cannot occur as the vertex link is a sphere; see Figure~\ref{fig:AAAfirstfinal}. In the other case, two squares can be surgered to give a non-separating torus; see Figure~\ref{fig:AAAsecondfinal}.

\begin{figure}[h]
  \begin{center}
      \includegraphics[height=3.7cm]{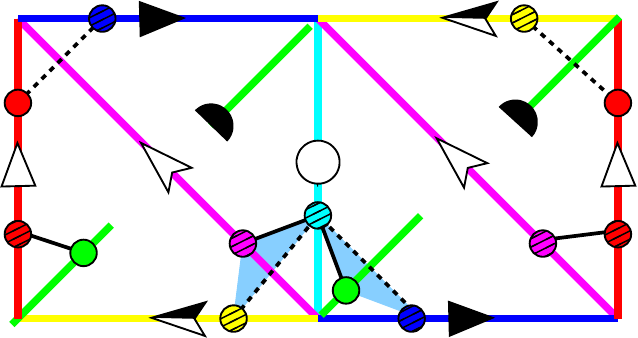}
       \qquad\qquad
         \includegraphics[height=3.7cm]{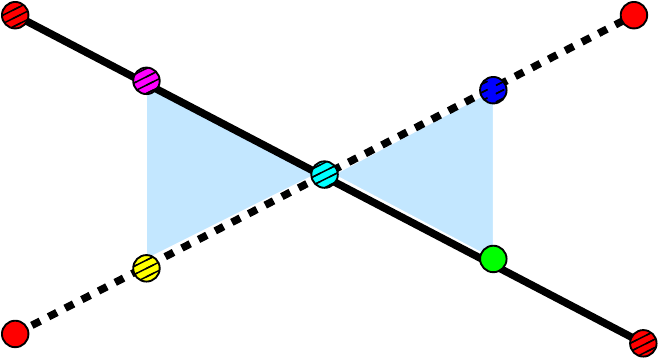}
          \end{center}
 \caption{Shown is one possible orientation for the companion of a tetrahedron of type AAA. The indicated normal arcs form two simple closed curves that meet transversely in one point on the vertex link. The transverse intersection can be seen by viewing the shaded triangles from the vertex link.} 
 \label{fig:AAAfirstfinal}
\end{figure}
\begin{figure}[h]
  \begin{center}
      \includegraphics[height=3.7cm]{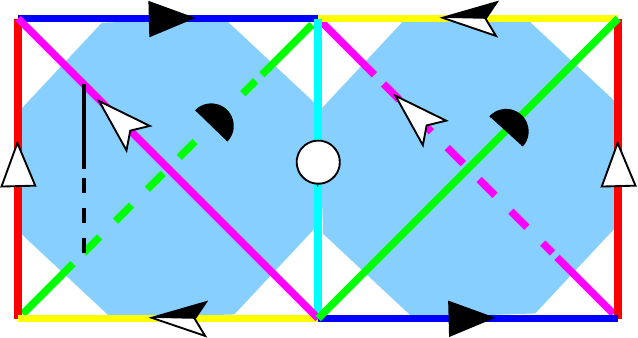}
          \end{center}
 \caption{Shown is the other possible orientation for the companion of a tetrahedron of type AAA. The capped 2--square torus can be viewed as made up of two octagons and two discs on a vertex link. The shown black arc meets the torus transversely and its endpoints can be joined on the vertex link by a path disjoint from the torus; thus certifying that the torus is non-separating.} 
 \label{fig:AAAsecondfinal}
\end{figure}

The four contradictory claims complete the proof of Theorem~\ref{thm:cluster}.
\end{proof}




\address{Feng Luo,\\ Department of Mathematics,\\ Rutgers University,\\ New Brunswick, NJ 08854, USA\\
(fluo@math.rutgers.edu)\\--}

\address{Stephan Tillmann,\\ School of Mathematics and Statistics F07,\\ The University of Sydney,\\ NSW 2006 Australia\\
(tillmann@maths.usyd.edu.au)}

\Addresses


\end{document}